\documentclass{amsart}
\usepackage{amssymb}
\usepackage{amsmath}
\usepackage{graphicx,color}
\usepackage{latexsym}
\usepackage{amsfonts}
\usepackage{mathdots}
\usepackage{eurosym}
\usepackage{enumerate}

\usepackage[active]{srcltx}

\newtheorem{theorem}{Theorem}[section]

\newtheorem{lemma}[theorem]{Lemma}
\newtheorem{corollary}[theorem]{Corollary}

\theoremstyle{definition}

\newtheorem{example}[theorem]{Example}

\newtheorem{remark}[theorem]{Remark}

\numberwithin{equation}{section}

\newcommand{\nC}{{\mathbb C}}
\newcommand{\nN}{{\mathbb N}}

\newcommand{\kip}{[\,\cdot\, , \cdot\,]}

\newcommand{\llf}{[\![} \newcommand{\rlf}{]\!]}
\newcommand{\lip}{\llf \,\cdot\, ,  \cdot\,\rlf}

\newcommand{\cK}{{\mathcal K}}
\newcommand{\cL}{{\mathcal L}}

\newcommand{\cN}{{\mathcal N}}

\newcommand{\fA}{{\mathfrak A}}
\newcommand{\fB}{{\mathfrak B}}
\newcommand{\fC}{{\mathfrak C}}
\newcommand{\fD}{{\mathfrak D}}

\newcommand{\fH}{{\mathfrak H}}

\newcommand{\fL}{{\mathfrak L}}

\newcommand{\wh}[1]{{#1}}
\newcommand{\V}[1]{\mathsf{#1}}
\newcommand{\mf}[1]{\mathfrak{#1}}

 \DeclareMathOperator{\lspan}{span}
 \DeclareMathOperator{\ospan}{span}

 \DeclareMathOperator{\diag}{diag}
 \DeclareMathOperator{\codim}{codim}
 \DeclareMathOperator{\ran}{ran}
 \DeclareMathOperator{\dom}{dom}
 \DeclareMathOperator{\mul}{mul}

 \DeclareMathOperator{\im}{Im}
 
 \DeclareMathOperator{\hol}{hol}
 \DeclareMathOperator{\extdeg}{extdeg}
  \DeclareMathOperator{\intdeg}{intdeg}
  \DeclareMathOperator{\rank}{rank}

\newcommand{\iu}{{\mathrm i}}

\begingroup
\def\2#1{\ifnum#1<10 0\fi\the#1}
\xdef\isodayandtime%
{\the\year-\2\month-\2\day\space\2{\count0}:\2{\count2}}

\endgroup

\newcounter{myenum}
  {\end{list}}

\newcounter{myenumi}
  {\end{list}}

\begin{document}

\title[On the reproducing kernel]{On the reproducing kernel of a Pontryagin space \\ of vector valued polynomials}
\author{B.~\'{C}urgus}
\address{Department of Mathematics, Western Washington University, Bellingham, WA 98225, USA} \email{\tt curgus@wwu.edu}
\author{A.~Dijksma}
\address{Johann Bernoulli Institute of Mathematics and Computer Science \\ University of Groningen \\
P.O. Box 407 \\
9700 AK Groningen, The Netherlands} \email{\tt
a.dijksma@rug.nl}


\begin{abstract} We give necessary and sufficient conditions under which the reproducing kernel of a Pontryagin space of $d \times 1$ vector polynomials is determined by a generalized Nevanlinna pair of $d \times d$ matrix polynomials.

\end{abstract}
\date{\today}

\maketitle

\section{Introduction}
\subsection{}
By Baire's category theorem, the Pontryagin space $\fB$ in the title is necessarily finite dimensional (see Remark~\ref{finitedim} below) and hence is a reproducing kernel space. Indeed, if $\bigl(\fB,\kip_\fB\bigr)$ is an $n$-dimensional Pontryagin space of $d \times 1$ vector polynomials and if $B(z)$ is a $d\times n$ matrix polynomial whose columns  $B_k(z)$, $k \in \{1,\ldots,n\}$, form a basis of $\fB$, then the reproducing kernel of $\mathfrak B$ is the $d \times d$ matrix polynomial in $z$ and $w^*$ given by
 \[
K(z,w)=B(z) G^{-1} B(w)^*, \quad z, w \in \nC,
 \]
where $G$ is the $n \times n$ Gram matrix associated with $B(z)$, that is,
 \[
G=[g_{jk}]_{j,k=1}^n, \quad g_{jk}=[ B_k,B_j ]_\fB, \quad j,k \in \{1,\ldots,n\}
 \]
(see \cite[Example 2.1.8]{abook} and the remark following it). The reproducing kernel of a reproducing kernel space is unique but can often be written in various ways. In this paper we give  necessary and sufficient conditions under which  $K(z,w)$ above is a  polynomial {\em Nevanlinna kernel}. This means that it can be written in the form
\begin{equation*} 
K(z,w) = K_{M,N}(z,w) :=  \dfrac{M(z) N(w)^*-N(z) M(w)^*}{z-w^*}, \ \ z, w \in \nC, \ z \neq w^*,
\end{equation*}
where $M(z)$ and $N(z)$ are $d \times d$ matrix polynomials such that
\begin{align} \nonumber
 & M(z) N(z^*)^*-N(z) M(z^*)^* = 0  \qquad
  \text{for all} \qquad  z \in \nC \\
\intertext{and}
   \label{irankz}
  & \quad \rank \bigl[ M(z) \, \ N(z)\bigr] = d  \qquad \text{for at least one} \qquad  z\in \nC.
\end{align}
If, in addition, the equality in \eqref{irankz} holds for all $z \in \mathbb C$, then the Nevanlinna kernel $K_{M,N}(z,w)$ is called a {\em full Nevanlinna kernel}.


The following theorem is the main result in this paper. It is proved in Section~\ref{s4}.
\begin{theorem} \label{tchpnk}
Let $\fB$ be a $($finite dimensional$)$ Pontryagin space of $d\times 1$ vector polynomials. Denote by $S_\fB$ the operator of multiplication by the independent variable in $\fB$ and by $E_\alpha$ the operator of evaluation at a point $\alpha \in \nC$. Then the reproducing kernel of $\fB$ is a polynomial Nevanlinna kernel if and only if the following two conditions hold:
\begin{enumerate}[{\rm (A)}]
\item \label{tchpnkca}
The operator $S_\fB$ is symmetric in $\fB$.
\item \label{tchpnkcb}
For some $\alpha \in \nC$ we have  $\ran \bigl(S_\fB -\alpha \bigr) = \fB \cap \ker E_{\alpha}$.
\end{enumerate}
 In this case the reproducing Nevanlinna kernel is full if and only if the equality in \eqref{tchpnkcb} holds for all $\alpha \in \mathbb C$.
\end{theorem}

We think Theorem~\ref{tchpnk} is new, possibly even in the positive definite case, that is, the case where the space $\mathfrak B$ is a reproducing kernel Hilbert space of vector polynomials. In that case $\mathfrak B$ in the theorem is a special case of  L.~de~Branges' Hilbert spaces of entire functions. For scalar functions, see \cite{dbbook}; for vector functions, see \cite{db1966} and \cite{dbr}. In particular, \cite[Theorems 1-3]{dbr} are closely related to Theorem~\ref{tchpnk}.  For results on the indefinite scalar case we refer to the series of papers on Pontryagin spaces of entire functions by M.~Kaltenb\"{a}ck and H.~Woracek. More specifically, \cite[Theorem 5.3]{kw19991} is closely related to Theorem~\ref{tchpnk} with $d=1$, \cite[Proposition 2.8]{kw2009} can be used to obtain a scalar version of Theorem~\ref{tchucsg} below, and \cite[Lemma 6.4]{kw19991} is linked with Theorem~\ref{tQ} in Section~\ref{s5}.
The emphasis in this paper is on vector polynomials and an indefinite setting.
\subsection{}
In the proof of Theorem~\ref{tchpnk} we use the following result which shows that the condition \eqref{tchpnkcb} in Theorem~\ref{tchpnk} completely determines the structure of $\mathfrak B$ as a linear space. We believe Theorem~\ref{tchucsg} is also new, but closely related to results around \cite[Proposition~2.3]{dv}. For the proof of Theorem~\ref{tchucsg} we refer to Section~\ref{s3}.
\begin{theorem} \label{tchucsg}
Let $\fB$ be a finite dimensional linear space of $d \times 1$ vector polynomials and let $\alpha \in \nC$.  The equality
\begin{equation} \label{eqSakg}
\ran \bigl(S_{\mf B} - \alpha \bigr) = \mf B \cap \ker E_{\alpha}
\end{equation}
holds if and only if there exist nonnegative integers $\mu_1, \ldots, \mu_d$ and a $d\times d$ matrix polynomial $W(z)$ with $\det W(\alpha) \neq 0$ such that the space $\fB$ consists of all vector polynomials of the form $W(z) \bigl[ p_1(z) \cdots p_d(z) \bigr]^\top$ where $p_j(z)$ runs through all scalar polynomials of degree strictly less than $\mu_j$, $j \in \{1,\ldots,d\}$.
The matrix $W(z)$ can be chosen such that
\begin{equation} \label{eqdetrannul}
\bigl\{\alpha \in \nC \,:\, \det W(\alpha) \neq 0 \bigr\} = \bigl\{\alpha \in \nC \,:\, \ran \bigl(S_{\mf B} - \alpha\bigr) = \mf B \cap \ker E_{\alpha}\bigr\}.
\end{equation}
\end{theorem}

\noindent It follows that the dimension of $\fB$ in Theorem~\ref{tchucsg} is  $\mu_1 + \cdots +\mu_d$. If the conditions \eqref{tchpnkca} and \eqref{tchpnkcb} of Theorem~\ref{tchpnk} hold, then the numbers $\mu_1, \ldots, \mu_d $ are the Forney indices of the block matrix polynomial $\bigl[ M(z) \ N(z)\bigr]$ corresponding to the reproducing Nevanlinna kernel of $\fB$.  Moreover, the defect numbers of $S_\fB$ coincide with the cardinality of the set \mbox{$\bigl\{j\in \{1,\ldots,d\} \,:\, \mu_j > 0\bigr\}$,} see Remarks~\ref{remdefect} and~\ref{rlast}. This offers a direct way of determining the dimension of the reproducing kernel space $\fB$ with reproducing Nevanlinna kernel $K_{M,N}(z,w)$ and the defect numbers of $S_\fB$ from the block matrix polynomial $\bigl[ M(z) \ N(z)\bigr]$.

\medskip

In the scalar case ($d=1$) the space $\mathfrak B$ in the above theorems is analogous  to the so-called Szeg\"o space, in the Hilbert space setting defined and studied in \cite{XL1, XL2} and in the Pontryagin space setting in \cite{amanuscript}. In the literature there are many papers characterizing special forms of the reproducing kernel of a reproducing kernel space. Of those related to a reproducing kernel Pontryagin space we mention \cite[Section~6]{ad} and \cite{a}. We refer to the references in these papers for papers dealing with the Hilbert space case. The characterizations in these works are often in terms of a special identity to be satisfied by the difference-quotient operator on the space. In some cases, such as in \cite[Theorem 4.1]{a} and \cite[Problems~51, Theorem~23]{dbbook} the invertibility of  $K(z,z)$ for some values of $z$ plays a role in proving the asserted representation of the kernel $K(z,w)$. We give in Section~\ref{s6} some examples where ${\det}K(z,w)=0$ for all $z,w \in \mathbb C$, see Example~\ref{example1} and Example~\ref{example2}.


\subsection{}
A pair $\{M(z), N(z)\}$ of $d \times d$ matrix functions $M(z)$ and $N(z)$ is called a {\em generalized Nevanlinna pair} if the functions are meromorphic on $\mathbb C \setminus \mathbb R$, the intersection of the domains of holomorphy ${\hol}(M)$ of $M(z)$ and ${\hol}(N)$ of $N(z)$ is symmetric with respect to the real axis,
\begin{align}
\label{iPneuG}
 & M(z) N(z^*)^*-N(z) M(z^*)^* = 0  \quad
  \text{for all} \quad  z \in  {\hol}(M)\cap {\hol}(N),\\
   \label{irankzG}
  & \rank \bigl[ M(z) \ \ N(z)\bigr] = d  \quad \text{for at least one} \quad  z\in {\hol}(M)\cap {\hol}(N),
\end{align}
and the {\em Nevanlinna kernel}
\begin{equation} \label{eqNevMNG}
K_{{M,N}}(z,w) :=  \dfrac{M(z)N(w)^*\!-\!N(z) M(w)^*}{z-w^*}, \, z, w \in {\hol}(M)\cap {\hol}(N), z \neq w^*,
\end{equation}
has a finite number of negative squares. Here, by a {\em finite number of negative squares} we mean that the set of numbers of negative eigenvalues counted according to multiplicity of the self-adjoint matrices of the form
 $$
\big[ \, x_j^* K_{M,N}(z_j,z_i)x_i \,\big]_{i,j=1}^n
 $$
with
 $$ n \in \mathbb N, \ x_i \in \nC^d, \   z_i \in  {\hol}(M)\cap {\hol}(N) , \  z_i\neq z_j^*, \  i,j \in \{1,\ldots, n\}$$
has a maximum. If this maximum is $\kappa$, then we say that the pair and the kernel have $\kappa$ negative squares. If $\kappa=0$ the adjective ``generalized'' is omitted; in that case the matrix functions are holomorphic at least on $\mathbb C \setminus \mathbb R$. The {\em number of positive squares} is defined in the same way.
 The pair and kernel are called {\em full} if the equality in \eqref{irankzG} holds for all $z \in {\hol}(M)\cap {\hol}(N)$.
If a (generalized) Nevanlinna pair $\{M(z), N(z)\}$ is such that $N(z)=I_d$, the $d \times d$ identity matrix, then it is identified with its first entry $M(z)$ and $M(z)$ is a ({\em generalized}) {\em Nevanlinna function}.

Nevanlinna pairs and generalized Nevanlinna pairs have been used in interpolation and moment problems (see \cite{kl1977}, \cite{abds1, abds2} and \cite{ade}), the description of generalized resolvents (see \cite{kl1971}) and in the theory of boundary value problems with eigenvalue dependent boundary conditions (see \cite{dls1984, dls1988}, \cite{cdr} and \cite{acdJFA}). Theorem~\ref{tchpnk} arose in our study \cite{acdNew} of an eigenvalue problem for an ordinary differential operator in a Hilbert space with boundary conditions which depend polynomially on the eigenvalue parameter. In that paper we linearize the original problem by extending the Hilbert space with a finite dimensional Pontryagin space of $d \times 1$ vector polynomials. This paper concerns the structure of such spaces.

\subsection{}
The Nevanlinna pair in a Nevanlinna kernel is not unique (see the paragraph before Example~\ref{example2}) and if $\{M(z), N(z)\}$ is a pair that determines the kernel, then the polynomial matrix $N(z)$ may be such that ${\det}N(z)=0$ for all $z \in \mathbb C$. In Section~\ref{s5}  we prove that one can always choose the pair so that ${\det}N(z)\not \equiv 0$ and the rational generalized Nevanlinna matrix function $N(z)^{-1}M(z)$ is essentially a $Q$-function of the symmetric operator $S_\mathfrak B$. We show that every self-adjoint extension of $S_\mathfrak B$ with nonempty resolvent set gives rise to a reproducing Nevanlinna kernel for the space $\mathfrak B$. The proof of Theorem~\ref{tchpnk} given in Section~\ref{s4} is  geometric, the proof of the first if statement in Theorem~\ref{tchpnk} given in Section~\ref{s5} is analytic. The last two examples in Section \ref{s6}, Example~\ref{example1} and Example~\ref{example2}, also serve to show that this analytic proof is constructive. In Section \ref{s6} we present three corollaries of Theorem \ref{tchpnk} and four examples.

\medskip

In Section~\ref{s2} we fix the notation related to vector and matrix polynomials and we recall the Smith normal form and the Forney indices of a matrix polynomial. Moreover, we prove some lemmas on the structure of a degenerate subspace of a finite dimensional Pontryagin space, the defect numbers of a simple symmetric relation in such a space and on polynomial Hermitian kernels. Although most proofs in this paper are based on methods from linear algebra, in the sequel we assume that the reader is familiar with (i) Pontryagin spaces and (multi-valued) operators on such spaces such as symmetric and self-adjoint relations (as in \cite{ikl}, \cite{as} and \cite{bdhs}), (ii) generalized Nevanlinna matrix functions (as in \cite{kl1977, kl1981}) and (iii) reproducing kernel Pontryagin spaces (as in \cite[Chapter 1]{adrs} and \cite[Chapter~7]{abook}).

The notion of a $Q$-function of a simple symmetric operator in a Pontryagin space is recalled in Section~\ref{s5}.

\medskip

{\bf Acknowledgements}. The authors thank Prof.~Marius van der Put for many discussions about Theorem~\ref{tchucs} and another proof of it and   Prof.~Daniel Alpay for acquainting them with his unpublished note \cite{amanuscript} which among other things lead to Remark~\ref{finitedim} and pointing out the references \cite{XL1, XL2}. We also thank the referee for useful comments.

\section{Notation and basic objects} \label{s2}

\subsection{}
The symbols $\mathbb N$, $\mathbb R$, and $\mathbb C$ denote the sets of positive integers, real numbers and  complex numbers. For $d \in \nN$  the vector space of all $d \times 1$ vectors is written as $\nC^d$ and $I_d$ stands for the
$d\times d$ identity matrix. The $k$-th row of $I_d$ will be denoted by $e_{d,k}$. For $k \in \{1,\ldots,n\}$ the subspace of $\nC^d$ spanned by $e_{d,1},\ldots,e_{d,k}$ will be called a {\em top coordinate subspace of} $\nC^d$; it will be denoted by $\nC^d_k$. The corresponding $d\times d$ projection matrix is denoted by $P_{d,k}$. We consider $\nC^d_0 = \{0\}$ a top coordinate subspace spanned by the empty set.

By $\nC^d[z]$ we denote the vector space over $\nC$ of all polynomials with coefficients in $\nC^d$. The space $\nC^d$ is identified with the subspace of all constant polynomials in $\nC^d[z]$. If $d=1$ we simply write $\mathbb C[z]$ and $\mathbb C$.
For $f \in \nC^d[z]\setminus\{0\}$ with
\[
f(z) = a_0 + a_1 z + \cdots + a_n z^n
\]
and for the zero polynomial $0$ we define
\begin{equation*}
\deg f = \max \bigl\{ k \in \{0,\ldots, n\} \,:\, a_k \neq 0 \bigr\}
\quad \text{and} \quad \deg 0  = - \infty.
\end{equation*}
Matrix polynomials are written as $B(z), M(z), N(z), \ldots$, that is, with their argument $z$; we use the bold face ${\mathbf P}(z), {\mathbf S}(z), \ldots$, for $d\times 2d$ matrix polynomials. Vector polynomials are sometimes written with and sometimes without their argument. The Fraktur alphabet $\fA, \fB, \fC,  \fH, \ldots$ is used to denote vector subspaces of $\nC^d[z]$. One exception to this is that $\fL$ will be used for a subspace of $\nC^{2d}[z]$.  An inner product on $\fB$  is denoted by $\kip_{\fB}$.  In a vector space, the symbol $\oplus$ denotes the direct sum of subspaces.
\begin{remark} \label{finitedim}
A Banach space with a countable Hamel basis is separable and hence, by \cite{L}, it is finite dimensional. Since \mbox{$\bigl\{z^n \,:\, n \in \{0\}\cup\nN\bigr\}$} is a countable Hamel basis of  ${\mathbb C}[z]$, the space $\nC^d[z]$ and all its subspaces also have  countable Hamel bases. Therefore any Pontryagin subspace of $\nC^d[z]$ is finite dimensional. In spite of this fact, to emphasize the finite dimensionality, we continue to speak of finite dimensional Pontryagin subspaces of $\nC^d[z]$.
\end{remark}
We introduce some special subspaces of $\nC^d[z]$. Let $n \in \{ 0\}\cup \nN$. The symbol $\nC^d[z]_{< n}$ stands for the set of all $f \in \nC^d[z]$ such that $\deg f < n$. In particular, $\nC^d[z]_{< 1}   = \nC^d$ and
$\nC^d[z]_{< 0}  = \{0\}$. A subspace $\fC$ of $\nC^d[z]$ is called {\em canonical} if there exist nonnegative integers $\mu_k, k \in \{1,\ldots,d\}$, such that
\begin{align*}
\fC & = \bigoplus_{k=1}^d  \bigl(\nC[z]_{< \mu_{k}}\bigr) e_{d,k}
\\ & =\bigl\{[p_1(z)\cdots p_d(z)]^\top \,:\, p_k(z) \in \mathbb C[z],\ {\deg}p_k<\mu_k,k\in \{1,\ldots,d\}\bigr\}.
\end{align*}
The numbers $\mu_1, \ldots, \mu_d$ will be called the {\em degrees of} $\fC$.   Without loss of generality we can assume that they are ordered:
$\mu_1 \geq \cdots \geq \mu_d \geq 0$. Then a canonical subspace is uniquely determined by its degrees. Clearly, the dimension of $\fC$ is the sum of its degrees.

Next we introduce some useful operators on $\nC^d[z]$. By ${P}_{d,k}$, $k \in \{1,\ldots,d\}$, we denote the natural extension of $P_{d,k}$ to $\nC^d[z]$,
by $\wh{S} \,:\, \nC^d[z] \to \nC^d[z]$ the operator of multiplication by  the independent variable, that is,
\begin{equation*}
(\wh{S}f)(z)  = z f(z), \quad f \in \nC^d[z],
\end{equation*}
and by $\wh{E}_\alpha \,:\, \nC^d[z] \to \nC^d$ the evaluation operator at the point $\alpha \in \nC$:
\begin{equation*}
\wh{E}_\alpha(f)  = f(\alpha), \quad f \in \nC^d[z].
\end{equation*}
It follows from the fundamental theorem of algebra that
\begin{equation} \label{eqrhS=khE}
\ran\bigl(\wh{S} - \alpha \bigr) = \ker E_\alpha.
\end{equation}
A wide class of operators on $\nC^d[z]$ is induced by $d \times d$ matrix polynomials. If $M(z)$ is such a polynomial we define the operator $\wh{M} : \nC^d[z] \to \nC^d[z]$ by
\[
\bigl(\wh{M}f\bigr)(z) = M(z) f(z), \quad f \in \nC^d[z].
\]
Clearly, $M S = S M$.
A square matrix polynomial is {\em unimodular} if its determinant is identically equal to a nonzero constant. If $M(z)$ is a unimodular matrix polynomial we will call $M$ a {\em unimodular} operator. In this case  ${M}$ is a bijection  and its inverse is also a unimodular operator.

\subsection{}
In the sequel we use that any nonzero $d \times n$  matrix polynomial $B(z)$  admits a {\em Smith normal form} representation (see for example \cite[Satz~6.3]{frg} or \cite{KTh}):
\begin{equation} \label{eqsnf}
B(z) = U(z) \begin{bmatrix} D(z) & 0 \\[4pt]
0 & 0
\end{bmatrix} V(z),
\end{equation}
where $U(z)$ is a $d\times d$  unimodular matrix polynomial, $V(z)$ is an $n\times n$  unimodular matrix polynomial and the matrix in the middle is a $d \times n$ matrix in which, for some $l \in \{1,\ldots,\min\{d,n\}\}$,  $D(z)$ is a diagonal $l\times l$ matrix polynomial with monic diagonal entries:  $D(z) = \diag \bigl(b_1(z),\ldots,b_l(z)\bigr)$ such that $b_i(z)$ is divisible by $b_{i+1}(z)$, $i \in \{1,\ldots,l-1\}$. Notice that $\rank B(\alpha) = l$ if and only if $b_1(\alpha) \neq 0$. If for some $z\in \nC$ the rank of  $B(z)$ is $d$ ($n$, respectively), then $l = d$ ($l = n$) and the zero block row (column) in the matrix in the middle of the right hand side in \eqref{eqsnf} is not present.

\begin{remark} \label{relD}
The matrix in the middle of the right hand side in \eqref{eqsnf} is uniquely determined by $B(z)$. In this paper $B(z)$ often is a matrix polynomial  whose columns form a basis of a subspace $\fB$ of $\nC^d[z]$. Then for any $d \times n$ matrix polynomial $B_1(z)$ whose columns also form a basis of $\fB$, the middle term of its Smith normal form is identical to that of $B(z)$. Thus, the number $l$ and the monic polynomials $b_j(z), j \in\{1,\ldots,l\}$,  above are uniquely determined by the subspace $\fB$ of $\nC^d[z]$.
\end{remark}

\subsection{}
Let ${\mathbf S}(z)$ be a $d\times 2d$ polynomial matrix. For $j\in \{1,\ldots,d\}$ let $\sigma_j$ be the degree of the $j$-th row of ${\mathbf S}(z)$. By definition, a degree of a row is the degree of its transpose. Define $\mathbf S_\infty$, the {\em internal degree} and the {\em external degree} of $\mathbf S(z)$ by:
\begin{align*}
{\mathbf S}_\infty &= \lim_{z\to \infty} \text{\small $\begin{bmatrix}
 z^{-\sigma_1} & \cdots & 0 \\
 \vdots & \ddots & \vdots  \\
  0 & \cdots & z^{-\sigma_d} \end{bmatrix}$ } {\mathbf S}(z),\\
\extdeg {\mathbf S}(z)& = \sigma_1 + \cdots + \sigma_d,\quad \text{and} \\[2mm]
\intdeg {\mathbf S}(z) &=\max \{\, \deg m(z)
\,:\, m(z)\ \text{is a}\ d\times d\ \text{minor of}\ \mathbf S(z)\}.
\end{align*}
For a proof of the following theorem we refer to \cite{McE}.
\begin{theorem} \label{tforney}
Let ${\mathbf P}(z)$ be a $d\times 2d$ matrix polynomial with ${\rank} \mathbf P(z)=d$ for all $z \in \mathbb C$. Let ${\mathbf S}(z)$ be a matrix polynomial in the family
\begin{equation} \label{eqfam}
\bigl\{{U}(z) {\mathbf P}(z)\,:\, \ {U}(z) \ \ \text{unimodular} \bigr\}.
\end{equation}
The following statements are equivalent:
\begin{enumerate}[{\rm (a)}]
\item \label{irr1}
\ $\extdeg {\mathbf S}(z) = \min \bigl\{ \extdeg U(z){\mathbf P}(z) \, : \, U(z) \ \text{unimodular} \bigr\}$.
\item \label{irr2}
\ ${\rank}{\mathbf S}_\infty=d$.
\item \label{irr3}
\ $\extdeg {\mathbf S}(z) = \intdeg {\mathbf S}(z)$.
\item \label{irr4}
\ ${\mathbf S}(z^*)^*$  has the ``predictable degree property''$:$ \\
For every \ $u(z) = \bigl[u_1(z) \ \cdots \ u_d(z) \bigr]^\top \in \nC^{d}[z]$ we have
\[
\deg \bigl(  {\mathbf S}(z^*)^* u(z) \bigr) = \max\bigl\{ \sigma_j + \deg u_j(z), \ j \in \{1,\dots,d\} \bigr\}.
\]
\end{enumerate}
\end{theorem}
A matrix polynomial $\mathbf S(z)$ in the family \eqref{eqfam} satisfying the conditions (\ref{irr1})--(\ref{irr4}) is called {\em row reduced}. The multiset $\{\sigma_1, \ldots, \sigma_d\}$ of row degrees for each row reduced matrix in the family  \eqref{eqfam} is the same.  Its elements are called the {\em Forney indices} of any of the matrices in the family \eqref{eqfam}, in particular of $\mathbf P(z)$. We extend this definition to the case where the $d \times 2d$ matrix polynomial $\mathbf P(z)$ has full rank for some $z \in \mathbb C$. For that we use the following
lemma which is a standard tool in system theory, see for
example \cite{F}.
\begin{lemma} \label{tief}
Let $\mathbf P(z)$ be a $d\times 2d$ matrix polynomial with ${\rank}\mathbf P(z)=d$
for some $z \in \mathbb C$. Then $\mathbf P(z)$ admits the factorization:
\begin{equation}\label{eqpgt}
\mathbf P(z) = G(z)\mathbf T(z) \ \ \ \text{for all} \ \ \ z \in \mathbb C,
\end{equation}
where $G(z)$ is a $d\times d$ matrix polynomial with $\det G(z)
\not\equiv 0$ and $\mathbf T(z)$ is a  $d\times 2d$ matrix polynomial with ${\rank}\mathbf T(z)=d$ for all $z \in \mathbb C$. This factorization is essentially
unique, meaning that if also
$
\mathbf P(z) =  G_1(z)\mathbf T _1(z)$ for all $z \in \mathbb C
$,
where $G_1(z)$ and $\mathbf T_1(z)$ have the same properties as $ G(z)$
and $\mathbf T(z)$, then for some unimodular $d\times d$ matrix polynomial
$E(z)$:
$
G_1(z) = G(z) E(z)^{-1}$ and $\mathbf T_1(z) =  E(z) \mathbf T(z)$, $z \in \mathbb C$.
\end{lemma}
The Forney indices of $\mathbf P(z)$ in the lemma
are by definition the Forney indices of the matrix
polynomial $\mathbf T(z)$ in the factorization \eqref{eqpgt}. By the second part of the lemma, this definition is independent of the choice of the matrix $G(z)$ in this factorization.

\medskip

For convenience of the reader we give a proof
of Lemma~\ref{tief} based on the Smith normal form of a matrix polynomial.
\begin{proof}[Proof of Lemma~{\rm~\ref{tief}}]
Let $\mathbf P(z)$ have the Smith normal form \eqref{eqsnf}. The assumptions imply that $l=d$ and that the matrix in the middle of \eqref{eqsnf} is equal to $\begin{bmatrix} D(z) & 0 \end{bmatrix}$.
Set $G(z) =
U(z)D(z)$ and $\mathbf T(z) =
\begin{bmatrix} I_{d} & 0 \end{bmatrix} \,  V(z)$. Then the factorization \eqref{eqpgt} holds and $G(z)$ and $\mathbf T(z)$ have the properties mentioned in the lemma.
To prove uniqueness we use the fact that, since $\mathbf T(z)$ and
$\mathbf T_1(z)$ have full rank for all $z\in\mathbb C$, they have right
inverses, see \cite{KTh}. These are $2d\times d$ matrix polynomials
$\mathbf S(z)$ and $\mathbf S_1(z)$ such that $\mathbf T(z) \mathbf S(z) = I_d$ and $\mathbf T_1(z)
\mathbf S_1(z) = I_d$ for all $z\in\mathbb C$. Define the matrix polynomials
$E(z) = \mathbf T_1(z) \mathbf S(z)$ and $ F(z) = \mathbf T(z) \mathbf S_1(z)$. Then the
equality $G(z)\mathbf T(z) =  G_1(z)\mathbf T_1(z)$ implies $E(z) =
G_1(z)^{-1}G(z)$ and $F(z) = G(z)^{-1}G_1(z)$, hence
$ E(z) F(z) = I_d$ for all but finitely many $z \in \mathbb C$. By
continuity the last equality holds for all $z \in \mathbb C$, hence
$E(z)$ is unimodular and has the stated properties.
\end{proof}

\medskip

\subsection{} The next two lemmas concern finite dimensional Pontryagin spaces. By the {\em positive {\rm (}negative{\rm )} index} of a Pontryagin space $\cK$ we mean the dimension of a maximal positive (negative) subspace of $\cK$; evidently, the dimension of $\cK$ is equal to the sum of the indices.
\begin{lemma} \label{lfdPs}
Let $\cK$ be a Pontryagin space with positive and negative index equal to $n$. Let $\cL$ be a subspace of $\cK$ with $\dim \cL = 2n - \tau$.  If $\cL$ contains a maximal neutral subspace of $\cK$, then $\cL^{\perp}$ is the isotropic part of $\mathcal L$ and  $\cL\text{\Large $/$}\!\cL^{\perp}$ is a Pontryagin space with positive and negative index equal to $n-\tau$.
\end{lemma}

\begin{proof}
Let $\cN$ be a maximal neutral subspace contained in $\cL$. Since $\cN^{\perp} = \cN$, the inclusion $\cN \subseteq \cL$, yields $\cL^{\perp} \subset \cN \subset \cL$. Therefore, $\cL^{\perp}$ is the isotropic part of $\mathcal L$ and $\dim \cL^{\perp} = \tau$.  Let
$\cL = \cL^{\perp} + \cL_- + \cL_+ $
be a pseudo-fundamental decomposition of $\cL$. Since $\cN$ is a
neutral subspace of $\cL$, we have $n = \dim \cN \leq \tau + \dim
\fL_{\pm}$. Therefore
\begin{equation*}
2n-\tau  = \dim \cL \\
  = \tau + \dim \cL_- + \dim \cL_+ \\
   \geq \tau + n - \tau + n - \tau \\
   = 2n-\tau.
\end{equation*}
This proves that $\dim \cL_- = \dim \cL_+ = n-\tau$.
\end{proof}

Recall that a symmetric relation $S$ in a Pontryagin space $\mathcal K$ is {\it simple} if $S$ has no non-real eigenvalues and
$\mathcal K=\overline{\ospan}\{ {\ker}(S^*-z)\,:\, z \in \mathbb C\setminus \mathbb R\}$.
Below $\mul S^*$ stands for the multi-valued part of the adjoint $S^*$ of $S$: $\mul S^*=\{ g \in \mathcal K \,:\, \{0,g\} \in S^*\}.$
\begin{lemma} \label{lsr}
Let $S$ be a simple symmetric relation in a finite dimensional
Pontryagin space of dimension $n$. Then the spaces $\mul S^*, \ker
S^*$, and $S^*\cap zI, z \in \nC$, have the same dimension $d'$, say.
In particular, the defect numbers of $S$ are both equal to $d'$.
Furthermore, $\dim {\ran} S=\dim S =\dim {\dom}S =  n-d'$ and $\dim S^* = n + d'$.
\end{lemma}
\begin{proof}
First notice that by \cite[Proposition~2.4]{acdJFA} $S$ is an operator
and $S$ has no eigenvalues. The following statements are equivalent:
\begin{enumerate}[(a)]
\item  \label{eqa}
\ $\dim\bigl(\mul S^*\bigr) = d'$.
\item \label{eqb}
\ $\codim\bigl(\dom S\bigr) = d'$.
\item  \label{eqc}
\ $\codim\bigl(\ran(S-z^*)\bigr) = d'$ for all $z\in \nC$.
\item \label{eqd}
\ $\dim\bigl(S^*\cap zI\bigr) =d'$ for all $z\in \nC$.
\end{enumerate}
The relation $(\dom S)^{\perp} = \mul S^*$ implies the equivalence
(\ref{eqa})$\Leftrightarrow$(\ref{eqb}). The equivalence
(\ref{eqb})$\Leftrightarrow$(\ref{eqc}) follows from the fact that
$S-z^*$ is one-to-one. By taking the orthogonal complements we obtain
the equivalence (\ref{eqc})$\Leftrightarrow$(\ref{eqd}). Notice that
(\ref{eqd}) with $z = 0$ implies that $d' = \dim\bigl(\ker
S^*\bigr)$. The equalities  $ n-d' =\dim {\dom} S=\dim S= \dim {\ran} S$  follow from (\ref{eqb}) and
the fact that $S$ is an injective operator.
Since $\dim S^* = 2n - \dim S$ the
last equality follows.
\end{proof}

\subsection{}
A $d\times d$ matrix function $K(z,w)$ will be called a {\em polynomial Hermitian kernel} if it is a polynomial of two variables $z$ and $w^*$ and
$K(z,w)^* = K(w,z)$, $z, w \in \nC$.
This implies that the degree of $K(z,w)$ as a polynomial in $z$ equals the  degree of $K(z,w)$ as a polynomial in $w^*$. If we denote this common degree by $p-1$, then $K(z,w)$ can be expanded as
\begin{equation} \label{eqKAs}
K(z,w) = \sum_{j=0}^{p-1} \sum_{k=0}^{p-1} A_{jk} z^j w^{*k},  \qquad z,w \in \nC,
\end{equation}
where $A_{jk}, j,k \in \{0,\ldots,p-1\}$, are $d\times d$ matrices. Since $K(z,w)$ is a Hermitian kernel, the $dp \times dp$ block matrix
\begin{equation}\label{eqA}
A = \begin{bmatrix}
A_{00} & \cdots & A_{0,p-1} \\
\vdots & \ddots & \vdots \\
A_{p-1,0} & \cdots & A_{p-1,p-1} \\
\end{bmatrix}
\end{equation}
is self-adjoint. It also follows that the number of negative squares of $K(z,w)$ equals the number of negative eigenvalues of $A$ and the number of positive squares of $K(z,w)$ equals the number of positive eigenvalues of $A$.  The dimension of the reproducing kernel space corresponding to $K(z,w)$ is the rank of $A$. These observations are used in the proof of the following lemma.
\begin{lemma} \label{ldouble}
Let $K(z,w)$ be a $d\times d$ matrix polynomial Hermitian kernel of degree $p-1$. For $q \in \nN$ set
\[
L_q(z,w) = \iu\, (z^{q} - w^{*q}) K(z,w), \qquad z,w \in \nC.
\]
If $q \geq p$, then the positive and the negative index of the reproducing kernel Pontryagin space with kernel $L_q(z,w)$ are equal and coincide with the dimension of the reproducing kernel Pontryagin space with kernel $K(z,w)$.
\end{lemma}
\begin{proof} Write $K(z,w)$ in the form \eqref{eqKAs} and denote by $A$ the matrix \eqref{eqA}.
We calculate the coefficients of the matrix polynomial $L_q(z,w)$ for $q \geq p$:
{\allowdisplaybreaks \begin{align*}
L_q(z,w) & = \iu z^{q} \sum_{j=0}^{p-1} \sum_{k=0}^{p-1} A_{jk} z^j w^{*k} - \iu w^{*q} \sum_{j=0}^{p-1} \sum_{k=0}^{p-1} A_{jk} z^j w^{*k} \\
& = \sum_{j=0}^{p-1} \sum_{k=0}^{p-1} \iu A_{jk} z^{q+j} w^{*k} + \sum_{j=0}^{p-1} \sum_{k=0}^{p-1} (-\iu) A_{jk} z^j w^{*(q+k)} \\
& = \sum_{j=0}^{q+p-1} \sum_{k=0}^{p-1} \iu A_{(j-q)k} z^{j} w^{*k} + \sum_{j=0}^{p-1} \sum_{k=0}^{q+p-1} (-\iu) A_{j(k-q)} z^j w^{*k} \\
& = \sum_{j=0}^{q+p-1} \sum_{k=0}^{q+p-1} \bigl(\iu A_{(j-q)k} -\iu A_{j(k-q)} \bigr) z^{j} w^{*k},
\end{align*}}
\!where we set $A_{jk} = 0$ whenever $j < 0$ or $k < 0$ or $j > p-1$ or $k > p-1$. In other words, the  $2d(p+q)\times 2d(p+q)$ self-adjoint matrix formed by the coefficients of $L_q(z,w)$ is given by
\begin{equation*} \label{eqantdi}
B = \begin{bmatrix}
0 & 0 & - \iu A \\
0 & 0 & 0 \\
\iu A & 0 & 0
\end{bmatrix}
\end{equation*}
where the $0$ in the center is a $d(q-p) \times d(q-p)$ matrix.
With
\[
E = \frac{1}{\sqrt{2}} \begin{bmatrix}
I_{dp} & 0 &  i I_{dp} \\ 0 & I_{d(q-p)} & 0 \\ i I_{dp} & 0 & I_{dp}
\end{bmatrix}
\]
we have $EE^* = I_{d(q+p)}$ and
\[
E^* B E = E^* \begin{bmatrix}
0 & 0 & - \iu A \\ 0 & 0 & 0  \\ \iu A & 0 & 0
\end{bmatrix}
E =
\begin{bmatrix}
A & 0 & 0  \\ 0 & 0 & 0 \\ 0 & 0 & -A
\end{bmatrix}.
\]
Therefore the rank of $B$ is twice the rank of $A$. Moreover, $B$ has equal numbers of positive and negative eigenvalues. Since the positive and negative index of the reproducing kernel Pontryagin space with kernel $L_q(z,w)$ coincide with the number of positive and negative eigenvalues of $B$ the lemma is proved.
\end{proof}

A polynomial reproducing Nevanlinna kernel introduced in the Introduction is a polynomial Hermitian kernel. Since in the proof of Theorem~\ref{tchpnk} the polynomials in a Nevanlinna pair never appear separate we adopt the following equivalent definition of a polynomial Nevanlinna kernel:
A $d\times d$ matrix function $K(z,w)$ is called a {\em polynomial  Nevanlinna kernel} if it can be represented as
\begin{equation} \label{eqpnkg}
{\mathbf P}(z) \V{Q}^{-1} {\mathbf P}(w)^* = \iu \, (z-w^{*}) K(z,w)  \quad \text{for all} \quad z, w \in \nC,
\end{equation}
where $\V{Q}$  is a $2d \times 2d$ self-adjoint matrix with $d$ positive and $d$ negative eigenvalues and $\mathbf P(z)$ is
a $d \times 2d$ matrix polynomial such that  $\mathbf P(z)$ has rank $d$ for some $z \in \nC$.
With
\begin{equation} \label{eqpnkgsQ}
\V{Q} = \V{Q}_1 := \begin{bmatrix} 0 & -\iu I_{_d} \\[5pt]  \iu I_{_d} & 0
\end{bmatrix} \quad \text{and} \quad {\mathbf P}(z) = \bigl[M(z) \ \ N(z)\bigr]
\end{equation}
the definition in the Introduction is obtained from the new one. The assumptions on $\V{Q}$ imply that there exists a constant invertible matrix $T$ such that $\V{Q} = T \V{Q}_1 T^*$. Now, if we write ${\mathbf P}(z) T = \bigl[M(z) \ \ N(z)\bigr]$, we have $K(z,w) = K_{M,N}(z,w)$.
Since ${\mathbf P}(z)$ is a polynomial, the condition that $\rank \mathbf P(z) = d$ for some $z \in \nC$ implies that $\rank {\mathbf P}(z) = d$ for all but finitely many $z \in \nC$. A polynomial Nevanlinna kernel will be called a {\em full} Nevanlinna kernel if ${\mathbf P}(z)$ can be chosen such that $\rank {\mathbf P}(z) = d$ for all $z\in \nC$.

\section{Proof of Theorem~{\rm~\ref{tchucsg}}} \label{s3}
\subsection{}
Let $\mf B$ be a vector subspace of $\nC^d[z]$. By $S_{\fB}$ we denote the range restriction of $\wh{S}$ to $\mf B$, that is,
$$
\dom S_{\fB}  =  \fB \cap \wh{S}^{-1}\mf{B}, \qquad
\bigl(S_{\mf B}f\bigr)(z)  = z f(z), \ f \in \dom S_{\fB}.
$$
In graph notation this means:
\[
S_\fB = \bigl\{ \{f,g\} \,:\, f, g \in \fB, \, g(z) = zf(z) \ \text{for all} \ z \in \nC\bigr\}.
\]
By \eqref{eqrhS=khE}, for $\alpha \in \nC$ we have
\begin{equation} \label{eqrannul}
\ran \bigl(S_{\fB} - \alpha\bigr) = \bigl(\wh{S} - \alpha\bigr) \bigl(\fB \cap \wh{S}^{-1}\mf{B}\bigr) \subseteq \ran \bigl(\wh{S} - \alpha\bigr) \cap \fB = \fB \cap \ker E_\alpha.
\end{equation}
The reverse inclusion is equivalent to the implication
\[
f \in \fB, \ \alpha \in \nC, \ f(\alpha) = 0 \quad \Rightarrow \quad f(z) = (z-\alpha) g(z) \ \ \text{for some} \ \ g \in \dom S_\fB.
\]
In some cases this implication does not hold. For  example, it does not hold for any $\alpha \in \nC$ in the space $\fB \subset \nC^2[z]$ given by
\[
\fB = \left\{\left[\begin{array}{c}
a_0 + a_2 z^2 \\ b_0+b_1 z
\end{array}\right] \, : \, a_0,a_2,b_0,b_1 \in \nC \right\}.
\]
Indeed, $\fB$ contains $\bigl[z^2-\alpha^2 \ \ z-\alpha\bigr]^\top$ which is $0$ at $z = \alpha$, but $\fB$ does not contain $\bigl[z+\alpha \ \ \ 1\bigr]^\top$.
That the implication, or equivalently, equality in \eqref{eqrannul} holds, is characterized in terms of canonical subspaces of $\nC^d[z]$ in Theorem~\ref{tchucsg} in the Introduction. This section is devoted to the proof of this theorem.

Let $B(z)$ be a $d\times n$ matrix polynomial whose columns form a basis for $\fB$, $n = \dim \fB$. Then, as will be shown in the proof of Theorem~\ref{tchucsg}, the sets in \eqref{eqdetrannul} are equal to $\bigl\{\alpha\in \nC \,:\, b_1(\alpha) \neq 0\bigr\}$, where $b_1(z)$ is the scalar polynomial in the Smith normal form \eqref{eqsnf} of $B(z)$.
We will first prove Theorem~\ref{tchucsg} for the case where the sets in \eqref{eqdetrannul} are equal to $\nC$, see Theorem~\ref{tchucs} below.
In this case $W(z)$ is unimodular.  The proof of Theorem~\ref{tchucs} is based on the following three lemmas.

\begin{lemma} \label{lSr}
Let $\mf B$ be a finite dimensional subspace of $\nC^d[z]$ such that
\begin{equation*} 
\ran \bigl(S_{\mf B} - \alpha\bigr) = \mf B \cap \ker E_{\alpha}  \quad
 \text{for all} \quad  \alpha \in \nC.
\end{equation*}
If $\dom S_{\fB} \subseteq \fB' \subseteq \fB$, then
$
\ran \bigl(S_{\fB'} - \alpha\bigr) = \fB'\cap  \ker E_{\alpha}$ for all $\alpha \in \nC$.
\end{lemma}
\begin{proof}
Let
 $
f \in \fB'\cap\,  \ker E_{\alpha}$. Then $f \in   \fB  \cap\, \ker E_{\alpha}=\ran \bigl(S_{\mf B} - \alpha\bigr)$, that is, $f=Sg-\alpha g$ for some $g \in {\dom}S_{\mathfrak B} \subseteq \mathfrak B'$. From $f,g \in \mathfrak B'$ we infer $g,Sg \in \mathfrak B'$. Hence
$g \in \dom S_{\fB'}$ and $f = (S_{\fB'}  -\alpha) g$. This proves
 $\fB' \cap \, \ker E_{\alpha}
\subseteq \ran \bigl(S_{\fB'} - \alpha\bigr)$.
Since the reverse inclusion is obvious, the lemma is proved.
\end{proof}

\begin{lemma} \label{lWtcs}
Let $\fB$ be an $n$-dimensional subspace of $\nC^d[z]$. Then
\begin{equation} \label{eqWtcs}
\fB \cap \ker \wh{E}_\alpha = \{0\}  \qquad \text{for all} \quad \alpha \in \nC
\end{equation}
if and only if there exists a unimodular operator $\wh{W}$ such that $\fB = \wh{W} \nC^d_{n}$, where $\nC^d_{n}$ is a top coordinate subspace of $\nC^d$.
\end{lemma}
\begin{proof}
If $n = 0$, the statements are trivial with $W(z) = I_d$. From now on we  assume $n \geq 1$. If $B(z)$ is any $d \times n$ matrix polynomial whose columns form a basis of $\fB$, then, clearly,
\begin{equation} \label{eqrankrannul}
\bigl\{\alpha \in \nC \,:\, \rank B(\alpha) = n \bigr\} = \bigl\{\alpha \in \nC \,:\, \fB \cap \ker \wh{E}_\alpha = \{0\} \bigr\}.
\end{equation}

Assume \eqref{eqWtcs}.  Let $B(z)$ be a $d \times n$ matrix polynomial whose columns form a basis of $\fB$. By \eqref{eqrankrannul}, for all $\alpha \in \nC$ the rank of $B(\alpha)$ is $n$ and $n \leq d$. Hence $B(z)$ admits the Smith normal form (see \eqref{eqsnf}):
$B(z) = U(z) \begin{bmatrix} I_n & 0\end{bmatrix}^\top V(z)$,
where $U(z)$ and $V(z)$ are unimodular. Define
\begin{equation} \label{eqextW}
W(z) = U(z) \begin{bmatrix} V(z) & 0 \\[4pt]
0 & I_{d-n} \end{bmatrix}.
\end{equation}
Then $W(z)$ is a unimodular $d\times d$ matrix polynomial and from $B(z) = W(z)\bigl[I_n \ \ 0\bigr]^\top$ it follows that $\fB = W\nC^d_n$. This proves the only if statement.

To prove the if statement, assume that there exists a $d\times d$ unimodular matrix polynomial $W(z)$ such that $\fB = W\nC^d_n$, where $\nC^d_{n}$ is a top coordinate subspace of $\nC^d$.
Then the columns of $B(z) = W(z) \bigl[I_n \ \ 0\bigr]^\top$ form a basis of $\fB$ and the rank of  $B(\alpha)$ is $n$ for all $\alpha \in \nC$. The equality \eqref{eqWtcs} follows from \eqref{eqrankrannul}.
\end{proof}

\begin{lemma} \label{lbig}
Let $\fB$ be an $n$-dimensional subspace of $\nC^d[z]$ and let $\fC$ be a canonical subspace of $\nC^d[z]$ with degrees $\mu_1 \geq \cdots \geq \mu_d \geq 0$ of which $k$ are positive. Assume
$
\fC + \wh{S}\fC  \subseteq \fB
$ and
\begin{equation} \label{eqWcs1}
\fB \cap \ker \wh{E}_\alpha \subseteq \fC + \wh{S}\fC \qquad \text{for all} \quad \alpha \in \nC.
\end{equation}
Then there exists a unimodular operator ${W}$ which acts as the identity on $\fC + \wh{S}\fC$ and is such that
\begin{equation} \label{eqB=Wcs}
\fB = \wh{W} \bigl(\nC^d_{m} + \fC + \wh{S}\fC \bigr),
\end{equation}
where $m = n - (\mu_1 + \cdots +\mu_d)(\geq 0)$.
\end{lemma}
Notice that $\nC^d_{m} + \fC + \wh{S}\fC$ is a canonical subspace. If $m \leq k$, then $\nC^d_{m} + \fC + \wh{S}\fC$ coincides with $\fC + \wh{S}\fC$ and $W = I_d$.

\begin{proof}
If $\fC =\{0\}$ the statement follows from Lemma~\ref{lWtcs}. From now on we assume $\fC \neq \{0\}$. Then $\mu_1 > 0$, consequently $k \in \{1,\dots,d\}$ and $\nC^d_k \subseteq \fC$. We consider two cases: $k=d$ and $k<d$.

(i) Assume $k=d$. Then $\nC^d \subseteq \fC \subseteq \fB$. Let $f \in \fB$. It can be written as
 $
f(z) = f(0) + z h(z)= f(0) + (\wh{S}h)(z)$.
Then $\wh{S}h = f - f(0) \in \fB$. Since $\bigl(\wh{S}h\bigr)(0) = 0$, by \eqref{eqWcs1} we get
 $
\wh{S}h \in \fB \cap \ker E_0 \subseteq  \fC + \wh{S}\fC$,
which implies $f = f(0) + \wh{S}h \in \fC + \wh{S}\fC$. That is, $\fB = \fC + \wh{S}\fC$. In this case $m = d$ and with $\wh{W} = \wh{I}_d$ the lemma is proved.

(ii) Assume $k < d$. If $\fC + \wh{S}\fC = \fB$, then \eqref{eqB=Wcs} holds with $W = I_d$ and $m = k$, implying that $\nC^d_m \subseteq \fC$. From now on we assume that $\fC + \wh{S}\fC$ is a proper subspace of $\fB$.
Recall that $\wh{P}_{d,k}$ is the coordinate projection. A trivial, but important observation is
\begin{equation} \label{eqtip0}
\wh{E}_{\alpha} \bigl(\fC + \wh{S}\fC\bigr) = \nC^d_k = \ran \wh{P}_{d,k} \quad \text{for all} \quad \alpha \in \nC.
\end{equation}
Let $\alpha \in \nC$ be arbitrary and let $f \in \fB$ be such that $( I_d - P_{d,k})f(\alpha) = 0$. By  \eqref{eqtip0}, there exists a $p \in \fC + \wh{S}\fC$ such that
$ p(\alpha) = P_{d,k} f (\alpha)$,
hence
\[
(f-p)(\alpha) = ( I_d - P_{d,k})f(\alpha) +  P_{d,k} f(\alpha) - p(\alpha) = 0,
\]
that is, $f-p \in \ker E_\alpha$. Since also $f-p \in \fB$, \eqref{eqWcs1} implies $f-p \in \fC +  S\fC$.  Thus both $p$ and $f-p$ belong to $\fC + \wh{S}\fC$, implying that $f \in \fC + S\fC$. We have proved the implication:
\begin{equation} \label{eqimp1}
f\in\fB, \quad \alpha \in \nC \quad \text{and} \quad ( I_d -  P_{d,k})f(\alpha) = 0 \quad \Rightarrow \quad  f \in \fC +  S\fC.
\end{equation}

Let $\fL_0$ be a subspace of $\fB$ be such that
\begin{equation*} \label{eq2c}
 \bigl(\fC + \wh{S}\fC \bigr) \cap \fL_0 = \{0\}  \qquad \text{and}  \qquad
 \fB = \bigl(\fC + \wh{S}\fC \bigr) \oplus \fL_0.
\end{equation*}
The dimension of $\fL_0$ is
\begin{equation*} \label{eqdimB}
j = n - \bigl(\mu_1+\cdots+\mu_d + k \bigr)\geq 1.
\end{equation*}
Let $B_0(z)$ be a $d \times j$ matrix polynomial whose columns form a basis of $\fL_0$. Decompose $B_0(z)$ as
\[
B_0(z) = \begin{bmatrix} B_{0,t}(z) \\[3pt]  B_{0,b}(z)
\end{bmatrix},
\]
where $B_{0,t}(z)$ is a $k\times j$ matrix polynomial and $B_{0,b}(z)$
a $(d-k)\times j$ matrix polynomial.
We will prove that
\begin{equation} \label{eqsr=br}
\rank B_{0,b}(\alpha) = \rank (I_d - {P}_{d,k})B_0(\alpha) = j \quad \text{for all} \quad \alpha \in \nC.
\end{equation}
The first equality is trivial. To prove the second let $\alpha \in \nC$ be arbitrary and $x \in \nC^j$ be such that $({I}_d - {P}_{d,k})B_0(\alpha)x = 0$.  Set $f(z) = B_0(z)x$.  Then $f \in \fL_0$ and  $({I}_d - {P}_{d,k})f(\alpha) = 0$.  By \eqref{eqimp1}, $f \in \bigl(\fC + \wh{S}\fC \bigr) \cap \fL_0$, consequently $f = 0$, that is, $B_0(z)x = 0$ for all $z \in \nC$. Since the columns of $B_0(z)$ form a basis of $\fL_0$, this implies $x=0$. This proves \eqref{eqsr=br}. Hence $j \leq d-k$. If $j=d-k$, the $(d-k)\times (d-k)$ matrix polynomial $W_b(z):=B_{0,b}(z)$ is unimodular. If $j<d-k$ we can extend $B_{0,b}(z)$ to a unimodular $(d-k)\times (d-k)$ matrix polynomial (also denoted by) $W_{b}(z)$ with $\det W_{b}(\alpha) \neq 0$ in the same way as the matrix $B(z)$ was extended to $W(z)$ in \eqref{eqextW} by the means of the Smith normal form. In both cases the first $j$ columns of $W_{b}(z)$ are the columns of $B_{0,b}(z)$.

Let $W_{t}(z)$ be the $k\times(d-k)$ matrix obtained from the $k\times j$ matrix $B_{0,t}(z)$ by adding $d-k-j$ zero columns on the right. Define the $d\times d$ matrix polynomial $W(z)$ by
\[
W(z) = \begin{bmatrix} I_k & W_t(z) \\[5pt] 0_{(d-k)\times k} & W_{b}(z)
\end{bmatrix}.
\]
Then $W(z)$ is unimodular and $W(z)e_{d,k+l}, l = 1,\ldots,j$, are the columns of the matrix $B_0(z)$. The operator $W$ acts as the identity on $\fC + \wh{S}\fC$ and $W \nC^d_{m} =\nC^d_k + \fL_0$, where
\[
m = k+ j = k+ n - \bigl(\mu_1+\cdots+\mu_d + k \bigr) = n - \bigl(\mu_1+\cdots+\mu_d \bigr).
\]
Hence
$
W \bigl( \nC^d_{m} + \fC + S\fC \bigr) = \nC^d_k + \fL_0 + \fC + S\fC = \fB.
$
\end{proof}

\begin{theorem} \label{tchucs}
Let $\fB$ be  a finite dimensional subspace of $\nC^d[z]$.  The equality
\begin{equation} \label{eqSak}
\ran \bigl(S_{\mf B} - \alpha\bigr) = \mf B \cap \ker E_{\alpha}  \quad
 \text{for all} \quad  \alpha \in \nC
\end{equation}
holds if and only if there exist a $d\times d$ unimodular matrix polynomial $W(z)$ and a canonical subspace $\fC$ of $\nC^d[z]$ such that  $\mf{B} = W \fC$.
\end{theorem}

\begin{proof}
We first prove the if statement. To prove \eqref{eqSak} it suffices to show that $$\fB \cap \ker E_\alpha \subseteq \ran (S_\fB -\alpha).$$ Let $f \in \fB \cap \ker E_\alpha$.  Then
$f(\alpha) = 0$ and $f = Wg$ for some $g \in \fC$. Since $W$ is unimodular,   $g(\alpha) = 0$. Since $\fC$ is canonical, the polynomial $g(z)/(z-\alpha)$ belongs to $\fC$. Therefore
 $
f(z)/(z-\alpha) = W(z)\bigl(g(z)/(z-\alpha)\bigr) \in \fB$,
  hence $f \in \ran(S_\fB- \alpha)$.

We prove the only if statement by induction on the dimension of $\fB$. Assume \eqref{eqSak}.
The theorem is obviously true if $\dim \fB = 0$. Lemma~\ref{lWtcs} implies that it is true if $\dim \fB = 1$ for then $\mathfrak B \cap E_\alpha=\{0\}$. Let $n \in \nN$ and state the inductive hypothesis:

If $\fA$ is a subspace of $\nC^d[z]$ with $\dim \fA < n$ and such that
\begin{equation} \label{eqSakA}
\ran \bigl(S_{\mf A} - \alpha\bigr) =  \mf A  \cap  \ker E_{\alpha}  \quad
 \text{for all} \quad  \alpha \in \nC,
\end{equation}
then there exists a unimodular $d \times d$ matrix polynomial operator $F(z)$ such that $F \mf{A}$ is a canonical subspace of $\nC^d[z]$.

Let $\fB$ be a finite dimensional subspace of $\nC^d[z]$ such that \eqref{eqSak} holds and $\dim\fB = n$. Then $\fA = \dom S_\fB$ is a proper subspace of $\fB$. Therefore $\dim \fA < n$. If $\fA =\{0\}$, then $\fB \cap \ker E_\alpha = \ran\bigl(S_\fB - \alpha\bigr) = \{0\}$ and the theorem follows from Lemma~\ref{lWtcs}. Now we assume $\fA \neq \{0\}$. By Lemma~\ref{lSr} the subspace $\fA$ satisfies \eqref{eqSakA}. By the inductive hypothesis there exists a unimodular matrix polynomial $F(z)$ such that $\fD:= F\fA$
is a canonical subspace of $\nC^d[z]$.
Since $F$ and $S$ commute we have
$
\fD = F \fA = F \dom S_\fB = \dom S_{F\fB}$,
hence
$
\fD + \wh{S} \fD \subseteq U \fB$.
To apply Lemma~\ref{lbig} to $F\fB$ we need to verify \eqref{eqWcs1}. Let $f \in \fB$ be such that $(Ff)(\alpha)=0$. Then $f(\alpha)=0$ and, by \eqref{eqSak}, there exists a $g \in \dom S_\fB =  \fA$ such that $f = S_\fB g - \alpha g \in \fA + \wh{S}\fA$.  Therefore, $Ff\in \fD + \wh{S} \fD$, which verifies \eqref{eqWcs1}. Lemma~\ref{lbig} applied to $F\fB$ yields that there exists a unimodular operator $U$ such that $U^{-1}F\fB$ is a canonical subspace of $\nC^d[z]$. This proves the theorem with $W=F^{-1}U$
\end{proof}
\subsection{}
The following lemma will be used to deduce Theorem~~\ref{tchucsg} from Theorem~~\ref{tchucs}.

\begin{lemma} \label{lonemany}
Let $\fB$ be an $n$-dimensional subspace of $\nC^d[z]$ and let $B(z)$ be a $d \times n$ matrix polynomial whose columns form a basis of $\fB$.
Let $l$ be the size of the square diagonal matrix in the Smith normal form \eqref{eqsnf} of $B(z)$. Then
\begin{equation} \label{eqrannulrank}
\bigl\{\alpha \in \nC \,:\, \ran(S_\fB - \alpha) = \fB \cap \ker E_\alpha \bigr\} = \bigl\{ \alpha \in \nC \,:\, \rank B(\alpha) = l \bigr\}
\end{equation}
if and only if the set on the left hand side is nonempty. In this case $\dim \ran S_{\mathfrak B}=\dim \mathfrak B - l$.
\end{lemma}
\begin{proof}
The only if statement follows from the fact that the set on the right hand side in \eqref{eqrannulrank} is nonempty. Before proving the if statement we show
\begin{equation} \label{eqSBl}
\dim \ran S_\fB \leq \dim \fB - l.
\end{equation}
For all $\alpha \in \nC$ we have
$
\ran(S_\fB - \alpha) \subseteq \fB \cap \ker E_\alpha
$,
and hence
\[
\dim \ran S_\fB  = \dim \ran(S_\fB - \alpha) \leq \dim (\fB \cap \ker E_\alpha)  = \dim \fB - \rank B(\alpha).
\]
Consequently, 
$
l = \max_{\alpha\in \nC} \rank B(\alpha) \leq \dim \fB - \dim \ran S_\fB
$.
This proves \eqref{eqSBl}.

To prove the if statement assume that $\alpha_0 \in \nC$ is in the set on the left hand side of \eqref{eqrannulrank}.  Then equality holds in \eqref{eqSBl}. Indeed, this follows from
 {\allowdisplaybreaks
\begin{align*}
\dim \fB - l & \geq \dim \ran S_\fB \\
 & = \dim \ran (S_\fB - \alpha_0 ) \\
 & = \dim \bigl( \fB \cap \ker E_{\alpha_0} \bigr) \\
 & = \dim \fB - \rank B(\alpha_0) \\
 & \geq \dim \fB - l.
\end{align*}}
\!\!This proves the last statement in the lemma. Now the equality \eqref{eqrannulrank} follows from the following sequence of equivalences which hold for all $\alpha \in \nC$:
 {\allowdisplaybreaks
\begin{align*}
\rank B(\alpha) = l \ & \Leftrightarrow \ \dim (\fB \cap \ker E_\alpha) = \dim \fB - l \\
 & \Leftrightarrow \ \dim (\fB \cap \ker E_\alpha) = \dim \ran S_\fB  \\
  &  \Leftrightarrow \  \dim (\fB \cap \ker E_\alpha) = \dim \ran (S_\fB -\alpha) \\
  & \Leftrightarrow \ \ran(S_\fB-\alpha) = \fB \cap \ker E_\alpha.  \qedhere
\end{align*}}
\end{proof}

\begin{proof}[Proof of Theorem~{\rm~\ref{tchucsg}}]
We first prove the if statement. It suffices to prove the inclusion $\fB \cap \ker E_{\alpha} \subseteq \ran (S_\fB -\alpha)$, as the reverse inclusion always holds. Let $f \in \fB \cap \ker E_{\alpha}$.  Then $f(\alpha) = 0$ and $f = Wg$ with $g \in \fC$. Since $W(\alpha)$ is invertible,  $g(\alpha) = 0$. As $\fC$ is canonical, the polynomial $h(z) = g(z)/(z-\alpha)$ belongs to $\fC$. Therefore $W(z) h(z) \in \fB$ and $(x-\alpha)W(z) h(z) = f(z)$, which implies $f \in \ran(S_\fB- \alpha)$.

To prove the only if statement, assume that \eqref{eqSakg} holds for $\alpha=\alpha_0$. Let $B(z)$ be a $d \times n$ matrix polynomial whose columns form a basis of $\fB$. Let
\begin{equation*} 
B(z) = U(z) \begin{bmatrix} D(z) & 0 \\[4pt]
0 & 0
\end{bmatrix} V(z)
\end{equation*}
be the Smith normal form \eqref{eqsnf} of $B(z)$ where $D(z)$ is an $l\times l$ diagonal matrix with nonzero diagonal entries. Now define the space $\fB_1 \subset \nC^d[z]$ as the span over $\nC$ of the columns of
\begin{equation*} 
B_1(z) = U(z) \begin{bmatrix} I_l & 0 \\[4pt]
0 & 0
\end{bmatrix} V(z).
\end{equation*}
Set
\[
F(z) = U(z) \begin{bmatrix} D(z) & 0 \\[4pt]
0 & I_{d-l}
\end{bmatrix} U(z)^{-1}.
\]
Then $\fB =  F\fB_1$ and ${\det}F(\alpha_0)\neq 0$. Moreover, since
$$\{\alpha \in \mathbb C \,:\,  {\det}F(\alpha)\neq 0\}=\{\alpha \in \mathbb C \,:\, {\rank}B(\alpha)=l\}$$
and by Lemma~\ref{lonemany}, \eqref{eqdetrannul} holds for $F(z)$.
From ${\det}F(\alpha_0)\neq 0$ it follows that
\[
\ran(S_\fB -\alpha_0) = \fB \cap \ker E_{\alpha_0} \quad \Rightarrow \quad  \ran(S_{\fB_1} -\alpha_0) = \fB_1 \cap \ker E_{\alpha_0}.
\]
Since $\rank B_1(\alpha) = l$ for all $\alpha \in \nC$, Lemma~\ref{lonemany} implies that
\[
\ran(S_{\fB_1} -\alpha) = \fB_1 \cap \ker E_{\alpha} \quad \text{for all} \quad \alpha \in \nC.
\]
By Theorem~\ref{tchucs}, there exists a unimodular matrix $U(z)$ such that $\mathfrak C=U^{-1}{\mathfrak B}_1$ is a canonical subspace of $\nC^d[z]$, hence $\fB = W\fC$ with $W=FU$. Finally, \eqref{eqdetrannul} holds, because $U$ is unimodular and $F(z)$ satisfies \eqref{eqdetrannul}.
\end{proof}
\subsection{}
Theorem~\ref{tchucsg} can also be formulated in terms of matrix polynomials:
\begin{theorem}\label{tmatrixversion}
Let $\fB$ be an $n$-dimensional subspace of $\nC^d[z]$, $n \geq 1$. Let $B(z)$ be a $d\times n$ matrix polynomial whose columns form a basis of $\fB$. Let $b_1(z)$ and $l$ be as in the Smith normal form \eqref{eqsnf} of $B(z)$. Then $l+\dim \dom S_\fB= \dim \fB$ if and only if there exist
\begin{enumerate}[{\rm (a)}]
\item \label{tmatva}
a $d \times d$ matrix polynomial $W(z)$ whose determinant has the same zeros as $b_1(z)$,
\item \label{tmatvb}
nonnegative integers $m$ and $\delta_0 \geq \delta_1 \geq \cdots \geq \delta_m$ with $\delta_0 + \cdots + \delta_m = n$ and
\item \label{tmatvc}
an invertible $n\times n$ constant matrix $T$
\end{enumerate}
such that
\begin{equation}\label{eqbwpt}
B(z) = W(z) \Bigl[ P_{\delta_0} \ \ P_{\delta_1} z \ \cdots \ P_{\delta_m} z^m \Bigr] \, T \quad \text{for all} \quad z \in \nC,
\end{equation}
where $P_\delta$ stands for the $d \times \delta$ matrix: $P_\delta=\begin{bmatrix} I_\delta &0 \end{bmatrix}^\top$.
\end{theorem}
\begin{proof} For all $\alpha \in \mathbb C$ we have ${\ker}(S_\mathfrak B-\alpha) \subseteq \mathfrak B \cap {\ker}E_\alpha$. For all $\alpha \in \mathbb C$ with $b_1(\alpha)\neq 0$ we have
 {\allowdisplaybreaks
\begin{align*}
 l+\dim \dom S_\fB & =\dim \ran \bigl( E_\alpha|_{\mathfrak B}\bigr) + \dim \ran(S_\mathfrak B -\alpha)
 \\ & \leq \dim \ran \bigl( E_\alpha|_{\mathfrak B}\bigr) + \dim
 (\mathfrak B \cap {\ker}E_\alpha)
\\ & =\dim \mathfrak B
\end{align*}}
and equality holds if and only if
${\ker}(S_\mathfrak B-\alpha)= \mathfrak B \cap {\ker}E_\alpha$.

To prove the only if statement, assume $l+\dim \dom S_\fB= \dim \fB$. Then we can apply Theorem~\ref{tchucsg}: There exist a matrix  polynomial $W(z)$ satisfying (\ref{tmatva}) and a canonical subspace $\mathfrak C$ of $C^d[z]$ such that $\mathfrak B=W\mathfrak C$. Let
$\mu_1\geq \cdots \geq \mu_d$ be the degrees of $\mathfrak C$. Since $n\geq 1$, we have $\mu_1\geq 1$. set $m=\mu_1-1$ and
$$\delta_j=\# \{i\in \{1,\ldots, d\}\,:\,  \mu_i>j\}, \quad j \in \{0, \ldots,m\}.$$
Then the equality in (\ref{tmatvb}) holds. Since the columns
of the matrix $\begin{bmatrix} P_{\delta_0} & \cdots & P_{\delta_m}z^m\end{bmatrix}$
form a basis for $\mathfrak C$, there exists a
matrix $T$ satisfying (\ref{tmatvc}) such that \eqref{eqbwpt} holds.

To prove the if statement, we note that (\ref{tmatva})-(\ref{tmatvc}) and \eqref{eqbwpt} imply that $\mathfrak B=W\mathfrak C$ with $\mathfrak C$ as above, and hence Theorem~\ref{tchucsg} can be applied and together with the if and only if statement at the beginning of the proof yield that
$l+\dim \dom S_\fB= \dim \fB$.
\end{proof}

\begin{remark} \label{rfim}
Denote by $(S^j)_\mathfrak B$ the range restriction of $S^j$ to $\mathfrak B$. Then in item (\ref{tmatvb}) of Theorem~\ref{tmatrixversion}: $m$ is the nonnegative integer with
$$\{0\}={\dom}(S^{m+1})_\mathfrak B \subsetneq  {\dom}(S^{m})_\mathfrak B$$
and
$$
\delta_j=\dim \dom (S^{j})_\mathfrak B-
\dim \dom (S^{j+1})_\mathfrak B, \quad j \in \{0,\ldots, m\}.
$$
Moreover, if we set $\delta_{-1}=d$, then the numbers
\begin{equation} \label{eqfir}
\mu_k=1+{\max}\bigl\{j \in \{-1,0,\ldots,m\}\,:\, \delta_j\geq k\bigr\}, \quad k\in \{1,\ldots,d\},
\end{equation}
are the degrees of the canonical space $W^{-1}\mathfrak B$. In the next section we will see that if $\fB \subset \nC^d[z]$ is a Pontryagin space which  satisfies the conditions (\ref{tchpnkca}) and~(\ref{tchpnkcb}) of Theorem~\ref{tchpnk}, then the numbers \eqref{eqfir} are the Forney indices of a matrix polynomial ${\mathbf P}(z)$ in a representation \eqref{eqpnkg} of the Nevanlinna reproducing kernel $K(z,w)$ of $\mathfrak B$; see Remark~\ref{rlast}.
\end{remark}

\section{Proof of Theorem~{\rm~\ref{tchpnk}}} \label{s4}
\subsection{}
We divide the proof of  Theorem~{\rm~\ref{tchpnk}} in two parts. In the first part we prove the if statements and in the second part we prove the only if statements. %
In the first part we will need characterizations of the defect numbers of the operator $S_{\mathfrak B}$ of multiplication by the independent variable in the Pontryagin space ${\mathfrak B}$ which are collected in the following remark.

\begin{remark} \label{remdefect}
Clearly, $S_{\mathfrak B}$ has no eigenvalues and for any subset $\Omega$ of $\mathbb C$ containing more than $d \times {\max}\{{\deg} f \,:\, f \in \mathfrak B\}$ elements we have $\cap_{w \in \Omega}\,{\ran}(S_{\mathfrak B}-w^*)=\{0\}$ or, equivalently,
$$
{\mathfrak B} = \lspan \bigl\{ \ker (S_{\mathfrak B}^*-w): w \in {\Omega} \bigr\}.
$$
Now assume \eqref{tchpnkca} of Theorem~{\rm~\ref{tchpnk}}. Then, by the above observations, $S_{\mathfrak B}$ is a simple symmetric operator and hence its defect numbers coincide and are equal to the codimension of ${\ran}S_{\mathfrak B}$, see Lemma~\ref{lsr}. It follows from Lemma ~\ref{lonemany} that the defect numbers of $S_{\mathfrak B}$ are also equal to the integer $l$ introduced in Remark~\ref{relD}. Hence $l\in \{1,\ldots,{\min}\{d,n\}\}$, where $n = \dim \fB$. Now also assume \eqref{tchpnkcb} of Theorem~{\rm~\ref{tchpnk}}. Then $l$ can be characterized in a different way. Indeed, by Theorem~\ref{tchucsg}, there exist a canonical subspace $\fC \subseteq \nC^d[z]$ with degrees $\mu_1 \geq \cdots \geq \mu_d \geq 0$ and a $d \times d$ matrix polynomial $W(z)$ with $\det W(\alpha) \neq 0$ such that $\fB  = W\fC$. Since, by Lemma~\ref{lsr}, we have $n-l={\dim}{\dom}S_{\mathfrak B} = {\dim} {\ran}S_{\mathfrak B}$ and since multiplication by $z$ and by $W(z)$ commute, $l$ is uniquely determined by the inequalities:
\begin{equation} \label{eqmus}
\mu_1 \geq \cdots \geq \mu_{l} \geq 1 \quad \text{and} \quad \mu_{l+1} = \cdots = \mu_d = 0.
\end{equation}
\end{remark}

\medskip

\begin{proof}[Proof of the if statements in Theorem~{\rm~\ref{tchpnk}}] Assume (\ref{tchpnkca}) and (\ref{tchpnkcb}).
We show that $\mathfrak B$ has a reproducing Nevanlinna kernel in steps (i)--(iv). In step (v) we prove the last if statement in the theorem.

\medskip

\noindent {\bf (i)} By Theorem~\ref{tchucsg} there exist a canonical subspace $\fC \subseteq \nC^d[z]$ with degrees $\mu_1 \geq \cdots \geq \mu_d \geq 0$ and a $d \times d$ matrix polynomial $W(z)$ with $\det W(\alpha) \neq 0$ such that $\fB  = W\fC$. Then, by Remark~\ref{remdefect}, the defect numbers of the symmetric operator $S_\fB$ are both equal to $l$, where $l$ is determined by
the inequalities \eqref{eqmus}.  It follows that the elements of $\mathfrak B$
are of the form:
\begin{equation}\label{eqx0}f(z) \in \mathfrak B \Rightarrow f(z)=W(z)\begin{bmatrix} x(z) \\ 0 \end{bmatrix},
\end{equation}
where $x(z)$ is an $l\times 1$ vector polynomial and $0$ denotes the zero vector of size $(d-l) \times 1$. Let $n=\dim \mathfrak B$ and
let $B(z)$ be a $d \times n$ matrix polynomial whose columns form a basis of $\fB$. Let $G$ be the corresponding Gram matrix and write the reproducing kernel $K(z,w)$ of $\fB$ as $K(z,w) =  B(z)  G^{-1}  B(w)^*$, $ z,w \in \nC$.
By (\ref{tchpnkcb}), this representation implies that for each $w \in \nC$ which belongs to the set in \eqref{eqrannulrank} the columns of $K(z,w)$ span an  $l$-dimensional subspace of $\mathfrak B$, in formula:
\begin{equation}\label{eqdimK}
{\dim}\bigl\{K(\,\cdot\,,w)x \,:\, x \in \mathbb C^d\bigr\}=l  \quad
\text{whenever} \quad w \in \bigl\{ \alpha \in \nC \,:\, \rank B(\alpha) = l \bigr\}.
\end{equation}

\medskip

\noindent {\bf (ii)} In the following we use graph notation in the space
$\fB \oplus\fB$. The operator $S_\fB$ is identified with its graph in $\fB\oplus\fB$ and its adjoint $S_\fB^*$ is the orthogonal complement of $S_\fB$ in $\fB \oplus\fB$ equipped with the Lagrange inner product \begin{equation*}
\bigl[\!\!\bigl[ \{f,g\},\{p,q\} \bigr]\!\!\bigr]
 = -\iu \bigl([g,p]_\fB - [f,q]_\fB \bigr), \qquad \{f,g\}, \{p,q\} \in \fB\oplus\fB.
\end{equation*}
Let $w \in \nC$, $x \in \nC^d$ and $\{f,Sf\} \in S_\fB$ be arbitrary. Then
\begin{align*}
\bigl[\!\!\bigl[ \{f,Sf\},\{K(\cdot,w)x,w^*K(\cdot,w)x\} \bigr]\!\!\bigr]
 & =  -\iu \bigl(\bigl[Sf,K(\cdot,w)x\bigr]_\fB
  - \bigl[f,w^*K(\cdot,w)x\bigr]_\fB \bigr) \\
 & =  -\iu \bigl(x^*(Sf)(w) - w x^*f(w) \bigr) \\
 & = 0
\end{align*}
and hence $\bigl\{ \{K(\cdot,w)x, w^* K(\cdot,w )x \}\, : \, x \in \nC^d \bigr\} \subseteq S_\fB^* \cap (w^* I)$ for all $w \in \nC$. According to the definition of defect number (see \cite[p.369]{acdJFA}) and by \eqref{eqdimK}, it follows that for all \mbox{$w \in \bigl\{ \alpha \in \nC\setminus \mathbb R \,:\, \rank B(\alpha) = l \bigr\}$}
\begin{equation} \label{eqdefsp}
\bigl\{ \{K(\cdot,w)x, w^* K(\cdot,w )x \} \,:\, x \in \nC^d \bigr\} = S_\fB^* \cap (w^* I),
\end{equation}
because for such $w$'s both sets have dimension $l$. Consider the subspace
\begin{equation} \label{eqlisub}
\fL_0 := \lspan \bigl\{ \{K(\cdot,w)x, w^*K(\cdot,w)x \} \,:\, w \in \nC,\ x\in \nC^d \bigr\}
\end{equation}
of $S_\fB^*$. Since $S_\fB$ has no eigenvalues, the generalized von Neumann formula given in  \cite[Theorem~3.7]{acdJFA} implies that for $d+1$ distinct points  $w_0,\ldots,w_d$ from the set  $\bigl\{ \alpha \in \nC \,:\, {\im} \alpha >0, \,\rank B(\alpha) = l \bigr\}$ we have
\begin{equation*}
S_\fB^* = S_\fB + S_\fB^* \cap (w_0^* I) + \sum_{j=1}^d S_\fB^* \cap (w_j I).
\end{equation*}
Combined with \eqref{eqdefsp} and \eqref{eqlisub} this yields
$S_\fB + \fL_0 \subseteq S_\fB^* \subseteq S_\fB + \fL_0$, and hence
\begin{equation} \label{eqB0S*}
S_\fB^* = S_\fB + \fL_0.
\end{equation}

\medskip

\noindent {\bf (iii)} Let $\bigl(\fB_1, \kip_{\fB_1} \bigr)$ be the reproducing kernel Pontryagin space whose kernel is
\[
L_1(z,w) = \iu \, (z-w^*)K(z,w), \quad z, w \in \nC.
\]
We claim that its positive and negative index are $l$. To prove the claim we consider the operator \mbox{$T: \bigl(S_\fB^*,\lip \bigr) \rightarrow \fB_1$}  defined by $T\bigl( \{f, g\} \bigr) = S f - g$, $\{f, g\} \in S^*_\mathfrak B$, and show that it is a partial isometry onto $\mathfrak B_1$ with null space ${\ker}T=S_\mathfrak B$. The last equality is easy to verify. That
$\ran T = \mathfrak B_1$  follows from \eqref{eqB0S*} as it implies (with $w\in \nC$ and $x \in \nC^d$):
\[
\bigl(T\bigl(\{K(\cdot,w)x,w^*K(\cdot,w)x\}\bigr)\bigr)(z) = (z - w^*)K(z,w)x = -\iu L_1(z,w)x.
\]
That $T$ is isometric follows from \eqref{eqB0S*}, the symmetry of $S_\mathfrak B$ and the equalities (with $w, v \in \nC$ and $x,y \in \nC^d$):
 {\allowdisplaybreaks
\begin{align*}
 \bigl[\!\!\bigl[ \bigl\{K(\cdot,w)x, & w^*K(\cdot,w)x\bigr\},  \bigl\{K(\cdot,v)y,v^*K(\cdot,v)y\bigr \} \bigr]\!\!\bigr]   \\
&   =  -\iu \Bigl(
  \bigl[w^*K(\cdot,w)x,K(\cdot,v)y\bigr]_\fB - \bigl[K(\cdot,w)x,v^*K(\cdot,v)y\bigr]_\fB
   \Bigr)  \\
 &  =  \iu \bigl( v-w^* \bigr) y^*K(v,w)x \\
 &  =  y^*L_1(v,w)x \\
 &  =  \bigl[ -\iu L_1(\cdot,w)x, -\iu L_1(\cdot,v)y\bigr]_{\fB_1} \\
 & = \Bigl[ T\bigl(\{K(\cdot,w)x,w^*K(\cdot,w)x\}\bigr),  T\bigl(\{K(\cdot,v)y,v^*K(\cdot,v)y\}\bigr) \Bigr]_{\fB_1}\!.
\end{align*}}
The claim now follows because
$\bigl(S_\fB^*/S_\fB,\lip \bigr)$ is a
Pontryagin space with positive and negative index $l$ (see \cite[Theorem 2.3(c)]{acdJFA}) and $T$ establishes a unitary mapping between this space and $\mathfrak B_1$.

\medskip

\noindent {\bf (iv)} Let ${B}_1(z)$ be a $d\times 2 l$ matrix polynomial whose columns form a basis of $\fB_1$, and let $\V{Q}_1$ be the corresponding $2 l \times 2 l$ Gram matrix. Then $\V{Q}_1$ is self-adjoint and, by the claim proved in (iii), has $l$ positive and $l$ negative eigenvalues.  Let ${B}_2(z)$ be the $d \times 2(d-l)$ matrix polynomial defined by
\begin{equation*}
{B}_2(z) = W(z) \begin{bmatrix}
0 & 0 \\ I_{d-l} & I_{d-l}
\end{bmatrix},
\end{equation*}
where the zero matrices are of size $l \times (d-l)$. Define the $d \times 2d$ matrix polynomial $\mathbf P(z)$ by $\mathbf P(z) = \bigl[ B_1(z) \ {B}_2(z) \bigr]$ and the $2d \times 2d$ block diagonal matrix $\V{Q}$ by
\[
\V{Q} = \begin{bmatrix}
\V{Q}_1 & 0  \\[3pt]
0 & \V{Q}_2
\end{bmatrix}, \quad \text{where} \quad \V{Q}_2 = \begin{bmatrix}
0 & \iu I_{d-l} \\[3pt]  - \iu I_{d-l} & 0
\end{bmatrix}.
\]
Then $\V{Q}$ is self-adjoint and has $d$ positive and $d$ negative eigenvalues. We claim that
\begin{enumerate}[{\rm (I)}]
\item \label{itrankP}
${\rank}\mathbf P (z)=d$ for some $z \in \mathbb C$, and
\item \label{itkernel}
$\mathbf P(z)\V{Q}^{-1} \mathbf P(w)^*=\iu (z-w^*)K(z,w)$ for all $z, w \in \nC$.
\end{enumerate}
We prove \eqref{itrankP}: The inclusion $\mathfrak B_1 =T(S^*_\mathfrak B)\subseteq \mathfrak B+S\mathfrak B$ and \eqref{eqx0} imply that there exists an $l \times 2l$ matrix polynomial $X(z)$ such that
$$B_1(z)=W(z)\begin{bmatrix} X(z) \\ 0 \end{bmatrix},$$
where now $0$ stands for the $(d-l)\times 2l$ zero matrix. The complex number $\alpha$ satisfying
\eqref{tchpnkcb} belongs to the sets in
\eqref{eqrannulrank} and \eqref{eqdetrannul} and hence
 {\allowdisplaybreaks
 \begin{align*}
{\rank}X(\alpha)&={\rank}B_1(\alpha)\\
&={\dim}E_\alpha \mathfrak B_1\\
&={\dim}{\lspan}\bigl\{ L_1(\alpha,w) x\,:\,  w\in \nC, x \in \nC^d \bigr\} \\
& = {\dim}{\lspan}\bigl\{ \iu (\alpha - w^*) K(\alpha,w) x\,:\, w\in \nC, x \in \nC^d \bigr\} \\
 & = {\dim}{\lspan} \bigl\{ K(\alpha,w) y \,:\, w \in \nC, y \in \nC^d \bigr\} \\
 & ={\dim}E_\alpha \fB \\
 &={\rank}B (\alpha)\\
 &=l.
\end{align*}}
The equality
$$
\mathbf P (z)=W(z) \begin{bmatrix} X(z) & 0 & 0\\ 0 &I_{d-l} & I_{d-l} \end{bmatrix}
$$
implies that ${\rank}\mathbf P(\alpha)=d$. This proves \eqref{itrankP}.
We prove \eqref{itkernel}:
\begin{align*} \mathbf P(z)\V{Q}^{-1}\mathbf P(w)^*&=
B_1(z)\V{Q}_1^{-1}B_1(w)^*+B_2(z)\V{Q}_2^{-1}B_2(w)^*\\
&= L_1(z,w)+ W(z) \begin{bmatrix}
0 & 0 \\ I_{d-l} & I_{d-l}
\end{bmatrix} \!\! \begin{bmatrix}
0 & \iu I_{d-l} \\  - \iu I_{d-l} & 0
\end{bmatrix} \!\!
 \begin{bmatrix}
0 & I_{d-l} \\ 0 & I_{d-l}
\end{bmatrix} W(w)^{*}\\
&=\iu (z-w^*) K(z,w).
\end{align*}
Items \eqref{itrankP}
and \eqref{itkernel} show that $K(z,w)$ is a polynomial Nevanlinna kernel for $\mathfrak B$.
This completes the proof of the if statement.

\medskip

\noindent {\bf (v)}
If \eqref{tchpnkcb} holds for all $\alpha \in \mathbb C$, then, by Theorem~\ref{tchucs}, $W(z)$ is unimodular and
the proof of \eqref{itrankP} shows that then ${\rank}\mathbf P(z)=d$ for all $z \in \mathbb C$.
\end{proof}

\subsection{}
In the proof of the only if statements in
Theorem~{\rm~\ref{tchpnk}} we use the following lemma.

\begin{lemma}\label{llp}
Let $\V{Q}$ be a self-adjoint $2d \times 2d$ matrix with $d$ positive and $d$ negative eigenvalues. Let $\mathbf P(z)$ be a $d \times 2d$ matrix polynomial such that
\begin{enumerate}[{\rm (a)}]
\item \label{iti} $\mathbf P(z)\V{Q}^{-1}\mathbf P(z^*)^*=0$ for all $z \in \mathbb C$,
\item \label{itii} ${\rank}\mathbf P(z)=d$ for all $z \in \mathbb C$,  and
\item \label{itiii} $\mathbf P(z)$ is row reduced and has row degrees $\sigma_1, \ldots, \sigma_d$, assumed ordered so that $\sigma_1\geq \cdots \geq \sigma_d$ and $\sigma_1={\deg}\mathbf P(z)=: p$.
\end{enumerate}
Equip $\mathbb C^{2d}[z]_{<p}$  with the inner product
 $$
[f,g]_\V{Q} =\sum_{j=0}^{p-1}  b_{p-1-j}^*\V{Q}^{-1} a_j,
\quad
f(z)=\sum_{j=0}^{p-1}a_jz^j,\ g(z)=\sum_{j=0}^{p-1}b_jz^j,\
a_j, b_j \in \nC^{2d},
 $$
and consider the following subspace of $\nC^{2d}[z]_{<p}:$
\begin{equation*}
\fL_p = \lspan \Biggl\{ \sum_{k=0}^{p-1} z^{p-1-k} w^{*k} {\mathbf P}(w)^*x \, :  \, w \in \nC, \ x \in \nC^d \Biggr\}.
\end{equation*}
Then the orthogonal complement of $\fL_p$ in $(\mathbb C^{2d}[z]_{<p}, \kip_{\V{Q}})$ is
\begin{equation} \label{eqNp}
\fL_p^{\perp}= \Bigl\{ f(z) \in \nC^{2d}[z]_{<p} \,:\, f(z) = {\mathbf P}(z^*)^* u(z) \ \text{with} \ u(z) \in \nC^{d}[z] \Bigr\}.
\end{equation}
It is the isotropic part of $\fL_p$ and
$\fL_p\text{\Large $/$}\!\fL_p^{\perp}$
is a Pontryagin space with positive and negative index $\sigma_1+\cdots+\sigma_d$.
\end{lemma}

\begin{proof}
For an element
 $f(z)=\sum_{j=0}^{p-1}a_jz^j\in \nC^{2d}[z]_{<p}$ the following equivalences hold:
 {\allowdisplaybreaks
\begin{align*}
f(z) \in \fL_p^{\perp}
  &   \Leftrightarrow  \
  \Biggl(\sum_{k=0}^{p-1} w^{*k} a_k^*\Biggr) \V{Q}^{-1} {\mathbf P}(w)^* = 0 \quad \text{for all} \ w \in \nC,  \\
 &
   \Leftrightarrow  \
  {\mathbf P}(z) \V{Q}^{-1} f(z)= 0 \quad \text{for all} \ z \in \nC,    \\
  &
  \Leftrightarrow  \ f(z)={\mathbf P}(z^*)^*u_z
\quad \text{for some} \ u_z \in \nC^d \ \text{and all} \ z \in \nC.    \\
\end{align*}}
The last equivalence follows from \eqref{iti} and \eqref{itii}. To prove that the vector $u_z$ depends polynomially on $z$
we use that the Smith normal form \eqref{eqsnf} of ${\mathbf P}(z)$ is given by: ${\mathbf P}(z) = U(z) \bigl[I_d \ \ 0 \bigr] V(z)$, where $U(z)$ and $V(z)$ are unimodular matrices. Then
$$f(z)={\mathbf P}(z^*)^*u_z=
V(z^*)^*\begin{bmatrix} I_d \\ 0 \end{bmatrix}
U(z^*)^*u_z \Rightarrow
u_z =  V(z^*)^{-*} \begin{bmatrix} I_d & 0 \end{bmatrix}U(z^*)^{-*} f(z)$$
and the right hand side belongs to $\mathbb C^d[z]$. This proves \eqref{eqNp}.

Since ${\mathbf P}(z^*)^*$ has full rank for every $z\in\nC$, it acts as an injection on $\nC^{d}[z]$, therefore
\begin{equation} \label{eqSp1}
{\dim}\fL_p^{\perp}=\dim\bigl\{ u(z) \in \nC^{d}[z] \,:\, \deg \bigl({\mathbf P}(z^*)^* u(z)\bigr) < p  \bigr\}.
\end{equation}
The number on the right hand side can be expressed in terms of the Forney indices of ${\mathbf P}(z)$. Indeed, since ${\mathbf P}(z)$ is row reduced, it has the ``predictable degree property'' (see Theorem~\ref{tforney}):
\[
\deg \bigl(  {\mathbf P}(z^*)^* u(z) \bigr) = \max\bigl\{ \sigma_j + \deg u_j(z) \,:\, j \in \{1,\dots,d\} \bigr\}.
\]
Consequently, the space on the right hand side of in \eqref{eqSp1} equals
\[
\bigl\{ u(z) \in \nC^{d}[z] \,:\, \deg u_j(z)  < p - \sigma_j, \  j \in \{1,\dots,d\} \bigr\},
\]
whose dimension is
$dp - \bigl(\sigma_1 + \cdots + \sigma_d\bigr).$
Hence $\dim \fL_p^{\perp} = dp - \bigl(\sigma_1 + \cdots + \sigma_d\bigr)$ and
\begin{equation*} \label{eqdimLp}
\dim \fL_p = \dim \nC^{2d}[z]_{<p}-\dim \fL_p^{\perp} = dp + \bigl(\sigma_1 + \cdots + \sigma_d\bigr).
\end{equation*}
To prove the last two statements of the lemma we apply Lemma~\ref{lfdPs} with $n=dp$ and $\tau=dp-(\sigma_1+\cdots+\sigma_d)$. The assumptions about $\V{Q}$ in the lemma readily imply that $\mathbb C^{2d}[z]_{<p}$ is a $2dp$-dimensional Pontryagin space with negative index $dp$. It remains to
construct a maximal neutral subspace of $\mathbb C^{2d}[z]_{<p}$ which is contained in $\fL_p$. We begin with the subspace ${\mathfrak H}={\ran}H$, where the  operator
$
H : \nC^{d}[z] \to \nC^{2d}[z]
$
maps $u(z) \in \nC^d[z]$ into the polynomial part of ${\mathbf P}\bigl({1}/{z^*}\bigr)^*u(z)$.
For example, if ${\mathbf P}(z)$ is written as:
\[
{\mathbf P}(z) = P_0 + z P_1 + \cdots + z^p P_p,
\]
then for $k \in \{0\}\cup\nN$ and  $x\in\nC^d$
\begin{equation*} 
{H}\bigl(z^k x\bigr) = \begin{cases}
 \displaystyle
 \sum_{j=0}^k z^{j}P_{k-j}^*x \  & \text{if} \quad  k < p, \\[18pt]
 \displaystyle z^{k-p} \sum_{j=0}^p z^{j}P_{p-j}^*x  \  & \text{if} \quad  p \leq k.
 \end{cases}
\end{equation*}
These formulas imply that $\mathfrak H$ is neutral
in $\bigl(\nC^{2d}[z]_{<p},\kip_{\V{Q}}\bigr)$.
Indeed, for $k, m \in \{0\}\cup\nN$ and  $x,y\in\nC^d$ we have
\[
\Bigl[{H}\bigl(z^{k} x \bigr), {H}\bigl(z^{m} y \bigr) \Bigr]_{\V{Q}}
 \!= \!\! \sum_{j = 0}^{\min\{k,m\}} y^* P_{p-1-m + j}\V{Q}^{-1} P_{k-j}^*x
  = \!\!\! \sum_{\substack{i+j=p-1-m+k\\i,j\in\{0,\ldots,p\}}}\!\! y^* P_{i}\V{Q}^{-1} P_{j}^*x
\]
and the last expression equals $0$ because the assumption \eqref{iti} is equivalent to
\begin{equation*} 
\sum_{\substack{j+k=n\\
j,k\in\{0,\ldots,p\}}} P_j\V{Q}^{-1} P_k^* = 0 \quad \text{for all} \quad n \in \{0,\ldots,2p\}.
\end{equation*}
Since, by \eqref{itii}, $P_0 = {\mathbf P}(0)$ has full rank, ${H}$ is degree preserving and hence injective.  Therefore, $\dim \mathfrak H = dp=(1/2)
\dim \bigl(\nC^{2d}[z]_{<p}\bigr)$ and $\mathfrak H$ is maximal neutral.

Define the mapping $R:\nC^{2d}[z]_{<p}\to\nC^{2d}[z]_{<p}$ by $(Rf)(z) = z^{p-1}f(1/z)$. Then $R$ is unitary with respect to $\kip_{\V{Q}}$ and hence $\mathfrak N := R\mathfrak H$ is also a maximal neutral subspace of $\bigl(\nC^{2d}[z]_{<p},\kip_{\V{Q}}\bigr)$. The proof of the lemma is complete if we show that $\mathfrak N \subseteq \fL_p$. For that we  consider the polynomials of the form
\begin{equation} \label{eqsps}
\sum_{k=0}^{2p-1} z^k w^{*k} x, \quad w \in \nC, \ x \in \nC^d.
\end{equation}
From
\[
{\mathbf P}\biggl(\frac{1}{z^*}\biggr)^* \left(\sum_{k=0}^{2p-1} z^k w^{*k} x \right) = \sum_{j=0}^{2p-1} \left(\frac{1}{z^j} \sum_{k=0}^{2p-1}z^k w^{*k} \right) P_j^* x
\]
and
\[
\frac{1}{z^j} \sum_{k=0}^{2p-1}z^k w^{*k} = \sum_{k=0}^{j-1} \frac{w^{*k}}{z^{j-k}} + w^{*j} \sum_{k=0}^{p-1}z^k w^{*k} + \sum_{k=p}^{2p-1-j}z^k w^{*(k+j)}
\]
we obtain
\begin{equation} \label{eqspf}
{H} \Biggl(\sum_{k=0}^{2p-1} z^k w^{*k} x\Biggr) = \sum_{k=0}^{p-1} z^k w^{*k} {\mathbf P}(w)^*x \ + \
  \text{higher  order terms}.
\end{equation}
Since the space $\nC^{d}[z]_{<2p}$ is spanned by polynomials in \eqref{eqsps}, each element of $\nC^{d}[z]_{<p}$ is also a sum of polynomials in \eqref{eqsps}. As  ${H}$ is degree preserving, the polynomials in $\mathfrak H = {H}\bigl( \nC^{d}[z]_{<p}\bigr)$ have degrees $<p$ and therefore they have the form \eqref{eqspf} with zero higher order terms. Thus
\[
 \mathfrak H \subset \lspan \left\{\sum_{k=0}^{p-1} z^k w^{*k} {\mathbf P}(w)^*x \,:\, \ w\in \nC, x \in \nC^d\right\}
\]
and $\mathfrak N=R \mathfrak H \subseteq \fL_p$. This proves Lemma~\ref{llp}.
\end{proof}

\begin{proof}[Proof of the only if statements in Theorem~{\rm~\ref{tchpnk}}]
Assume that the reproducing kernel of $\fB$ is a polynomial Nevanlinna kernel $K(z,w)$:
\begin{equation} \label{eqPneut}
\iu \, (z-w^{*}) K(z,w)  = {\mathbf P}(z) \V{Q}^{-1} {\mathbf P}(w)^*  \quad \text{for all} \quad z, w \in \nC,
\end{equation}
where $\V{Q}$ is a self-adjoint $2d \times 2d$ matrix with $d$ positive and $d$ negative eigenvalues and ${\mathbf P}(z)$ is a $d \times 2d$ matrix polynomial with ${\rank}\mathbf P(z)=d$ for some $z \in \mathbb C$.
Note that \eqref{eqPneut} implies \eqref{iti} of Lemma~\ref{llp}:
\begin{equation}\label{eqPneu}
\mathbf P(z)\V{Q}^{-1}\mathbf P(z^*)^*=0 \ \text{ for all}\ z \in \mathbb C.
\end{equation}
We prove \eqref{tchpnkca} and \eqref{tchpnkcb} in the steps (i)--(iv), in step (v) we prove the last only if statement in the theorem.

\medskip

\noindent {\bf
(i)} In this step we prove \eqref{tchpnkcb} under the assumption that \eqref{itii} and \eqref{itiii} of Lemma~\ref{llp} hold. Denote by $\mathfrak B_p$ the reproducing kernel Pontryagin space with kernel
$$L_p(z,w)  = \iu\, (z^{p}-w^{*p}) K(z,w)
 = \Biggl( \sum_{k=0}^{p-1} z^{p-1-k} w^{*k} \Biggr) {\mathbf P}(z) \V{Q}^{-1} {\mathbf P}(w)^*, \quad z,w \in \nC.
$$
Then
\begin{equation*}
\fB_p = \lspan \Biggl\{ {\mathbf P}(z) \V{Q}^{-1} \sum_{k=0}^{p-1} z^{p-1-k} w^{*k} {\mathbf P}(w)^*x \,:\,   w \in \nC, \ x \in \nC^d \Biggr\}
\end{equation*}
and
\[
\bigl[ {\mathbf P}(z) \V{Q}^{-1}f,{\mathbf P}(z) \V{Q}^{-1}g \bigr]_{\fB_p} = \sum_{k=0}^{p-1} \Bigl( v^{*(p-1-k)} {\mathbf P}(v)^* y\Bigr)^* \V{Q}^{-1} \Bigl( w^{*k} {\mathbf P}(w)^*x \Bigr),
\]
where
$$
f(z) = \sum_{k=0}^{p-1} z^{p-1-k} w^{*k} {\mathbf P}(w)^*x, \quad
 g(z) = \sum_{k=0}^{p-1} z^{p-1-k} v^{*k} {\mathbf P}(v)^*y.
$$

Comparing this inner product with the one defined in Lemma~\ref{llp}, we find that ${\mathbf P}(z) \V{Q}^{-1}$ considered as a multiplication operator maps
$\fL_p \subset \nC^{2d}[z]_{<p}$ isometrically onto $\fB_p$ and its null space is $\fL_p^\perp $ (see the second of the three equivalences in the beginning of the proof of Lemma~\ref{llp}). Hence, $\dim \fB_p = 2(\sigma_1+\cdots+\sigma_d)$ and the positive and the negative index of $\fB_p$ equal $\sigma_1+\cdots+\sigma_d$. According to Lemma~\ref{ldouble}, we have ${\dim}\fB =\sigma_1+\cdots+\sigma_d$.
The space $\fB$ is spanned by the columns of  $K(z,w), w\in\nC$ and for $j \in \{1,\ldots,d\}$ the degree of the $j$-th row of $K(z,w)$ as a polynomial in $z$ is equal to $\max \bigl\{ 0, \sigma_j-1\bigr\}$. Therefore
 $
\fB \subseteq \bigoplus_{j=1}^d \bigl( \nC[z]_{< \sigma_j} \bigr)e_{d,k}
 $.
Since both spaces have dimension $\sigma_1+\cdots+\sigma_d$, equality prevails:
\begin{equation} \label{eqBcan}
\fB_0 = \bigoplus_{j=1}^d \bigl(\nC[z]_{< \sigma_j}\bigr) e_{d,k}.
\end{equation}
This implies \eqref{tchpnkcb}.

\medskip

\noindent {\bf
(ii)} In this step we prove \eqref{tchpnkca} under the assumption that \eqref{itii} and \eqref{itiii} of Lemma~\ref{llp} hold. Set
\[
M(z) = \begin{bmatrix}
 z^{-\sigma_1} & \cdots & 0 \\
 \vdots & \ddots & \vdots  \\
  0 & \cdots & z^{-\sigma_d} \end{bmatrix}.
\]
Then ${\mathbf P}_\infty = \lim_{z\to \infty} M(z) {\mathbf P}(z)$  and by \eqref{eqPneu} we have
\begin{equation}\label{eq*}
{\mathbf P}_\infty\V{Q}^{-1} {\mathbf P}_\infty^* = \lim_{z\to \infty} M(z)  {\mathbf P}(z)\V{Q}^{-1}{\mathbf P}(z^*)^* D(z^*)^* =0.
\end{equation}
Since ${\mathbf P}_\infty$ has full rank, \eqref{eq*} implies that the linear span of the columns of ${\mathbf P}_\infty^*$ is a maximal neutral subspace of $\bigl(\nC^{2d},\kip_{\V{Q}}\bigr)$ and this span coincides with the null space of ${\mathbf P}_\infty\V{Q}^{-1}$.
We claim that for $a\in \nC^{2d}$
\begin{equation}\label{eq}
{\mathbf P}(z)a \in \fB \quad \Leftrightarrow \quad {\mathbf P}_\infty a=0.
\end{equation}
To prove the claim assume first that ${\mathbf P}(z)a \in \fB$. From \eqref{eqBcan} we see that the degree of the $j$-th entry of the vector polynomial ${\mathbf P}(z)a$ is strictly less than $\sigma_j, j\in\{1,\ldots,d\}$. Hence
$
{\mathbf P}_\infty a = \lim_{z\to \infty} M(z) \bigl( {\mathbf P}(z) a \bigr) = 0.
$
As to the converse, first notice that by the definition of ${\mathbf P}_\infty$ the row degrees of the matrix polynomial
$
{\mathbf P}_0(z) = {\mathbf P}(z) - M(z)^{-1} {\mathbf P}_\infty
$
are strictly less than $\sigma_j, j\in\{1.\ldots,d\}$. By \eqref{eqBcan} we have that ${\mathbf P}_0(z) a \in \fB$ for all $a \in \nC^{2d}$.
Now assume ${\mathbf P}_\infty a=0$. Then
\[
{\mathbf P}(z)a = {\mathbf P}_0(z)a + M(z)^{-1} {\mathbf P}_\infty a = {\mathbf P}_0(z)a \in \fB.
\]
This completes the proof of \eqref{eq}.

Consider $f \in \fB$. Since $\fB$ is finite dimensional it can be written as
\begin{equation} \label{eqfinB}
f(z) = \sum_{i=1}^m K(z,w_i)x_i, \quad m \in \nN, \   w_i \in \nC, \  x_i \in \nC^d,  \  i \in \{1,\ldots,m\}.
\end{equation}
The next sequence of equivalences follows from
\eqref{eq} and the observation after \eqref{eq*}:
{\allowdisplaybreaks
\begin{align*}
f \in \dom S_\fB & \quad \Leftrightarrow \quad  Sf \in \fB \\
 & \quad \Leftrightarrow \quad \sum_{i=1}^m (z-w_i^*) K(z,w_i) x_i \in \fB \\
 & \quad \Leftrightarrow \quad
 {\mathbf P}(z) \V{Q}^{-1}
  \Biggl( \sum_{i=1}^m  {\mathbf P}(w_i)^* x_i \Biggr) \in \fB \\
 & \quad \Leftrightarrow \quad
 {\mathbf P}_\infty \V{Q}^{-1}
  \Biggl( \sum_{i=1}^m  {\mathbf P}(w_i)^* x_i \Biggr) = 0\\
   & \quad \Leftrightarrow \quad
  \sum_{i=1}^m  {\mathbf P}(w_i)^* x_i  = {\mathbf P}_\infty^* x \quad \text{for some} \quad x \in \nC^d.
\end{align*}}
Let $f \in \fB$ be given by \eqref{eqfinB} and let $g \in \fB$ be of the form
\begin{equation*}
g(z) = \sum_{j=1}^n K(z,v_j)y_j, \quad n \in \nN, \  v_j \in \nC, \  y_j \in \nC^d, \ j \in \{1,\ldots,n\}.
\end{equation*}
Assume that $f, g \in \dom S_\fB$. Then there exist $x, y \in \nC^d$ such that
\[
\sum_{i=1}^m  {\mathbf P}(w_i)^* x_i  = {\mathbf P}_\infty^* x \quad \text{and} \quad \sum_{i=1}^m  {\mathbf P}(v_i)^* y_i  = {\mathbf P}_\infty^* y
\]
and using the reproducing kernel property of $K(z,w)$ we have
{\allowdisplaybreaks
\begin{align*}
\bigl[Sf,g\bigr]_\fB - \bigl[f,Sg\bigr]_\fB & = \sum_{j=1}^n \sum_{i=1}^m v_j y_j^* K(v_j,w_i) x_i - \left(\sum_{i=1}^m \sum_{j=1}^n w_i x_i^* K(w_i,v_j)y_j\right)^*\\
& = \sum_{j=1}^n \sum_{i=1}^m v_j y_j^* K(v_j,w_i) x_i - \sum_{i=1}^m \sum_{j=1}^n w_i^* y_j^* K(v_j,w_i)x_i\\
& = \sum_{j=1}^n \sum_{i=1}^m \bigl(v_j-w_i^*\bigr) y_j^* K(v_j,w_i) x_i\\
& = -\iu \, \sum_{j=1}^n \sum_{i=1}^m  y_j^* {\mathbf P}(v_j)\V{Q}^{-1}{\mathbf P}(w_i)^* x_i\\
& = -\iu \, \Biggl(\sum_{j=1}^n  y_j^* {\mathbf P}(v_j)\Biggr) \V{Q}^{-1} \Biggl( \sum_{i=1}^m {\mathbf P}(w_i)^* x_i \Biggr)\\
& = -\iu \, y^* {\mathbf P}_\infty \V{Q}^{-1} {\mathbf P}_\infty^* x \\
& = 0.
\end{align*}}
This proves that $S_\fB$ is symmetric.

\medskip

\noindent {\bf
(iii)} In this step we only assume \eqref{itii} of Lemma~\ref{llp}: ${\rank}\mathbf P(z)=d$ for all $z \in \mathbb C$.
Then there is a unimodular $d\times d$ matrix polynomial $U(z)$ such that ${\mathbf S}(z) = U(z){\mathbf P}(z)$ is row reduced with ordered row degrees $\sigma_1 \geq \cdots \geq \sigma_d$. Then $U$ is an isometry from $\fB$ onto the reproducing kernel Pontryagin space $\fC$ with kernel
\begin{equation}\label{eqSQS}
- \iu \, \dfrac{{\mathbf S}(z) \V{Q}^{-1} {\mathbf S}(w)^*}{z-w^*}.
\end{equation}
According to what has already been proved in (i)
\[
U\fB = \fC = \bigoplus_{j=1}^d \bigl(\nC[z]_{< \sigma_j}\bigr) e_{d,k}.
\]
Thus $\fB = U^{-1}\fC$ and, by Theorem~\ref{tchucs},
\eqref{tchpnkcb} holds for all $\alpha \in \mathbb C$. According to part (ii) of this proof, $S_{U\fB}$ is symmetric, hence $S_\fB = U^{-1} S_{U\fB} U$ is also symmetric, that is,
\eqref{tchpnkca} holds.

\medskip

\noindent {\bf
(iv)} Finally we prove that (\ref{tchpnkca}) and (\ref{tchpnkcb}) hold if ${\rank}\mathbf P(z)=d$ for some $z \in \nC$ as in the beginning of this proof. In that case there exist a $d\times d$ matrix polynomial $G(z)$ with ${\det}G(z) \not \equiv 0$ and a $d\times 2d$ matrix polynomial ${\mathbf S}(z)$ with ${\rank}\mathbf S(z)=d$ for all $z \in \mathbb C$ such that
${\mathbf P}(z) = G(z) {\mathbf S}(z)$ for all $z \in \nC$, see Lemma~\ref{tief}. If by $\fA$  we denote the reproducing kernel space with Nevanlinna kernel \eqref{eqSQS}, then, by what has been proved in (iii), the operator $S_\fA$ is symmetric and for almost all $\alpha \in \nC$ we have $\ran(S_\fA - \alpha) = \fA \cap \ker E_\alpha$.  Now (\ref{tchpnkca}) and (\ref{tchpnkcb}) follow since the multiplication operator $G$ corresponding to $G(z)$ is an isomorphism from $\fA$ onto~$\fB$.

\medskip

\noindent {\bf(v)} The last only if statement in the theorem follows from step (iii) above and  Theorem~\ref{tchucs}.
\end{proof}

\begin{remark} \label{rlast}
Assume that $\fB \subset \nC^d[z]$ is a Pontryagin space which satisfies the conditions (\ref{tchpnkca}) and~(\ref{tchpnkcb}) of Theorem~\ref{tchpnk}. Then,  by Theorem~\ref{tchpnk}, there is a generalized Nevanlinna pair $\{M(z), N(z)\}$ such that the $d\times 2d$ matrix polynomial ${\mathbf P}(z)=[M(z) N(z)]$ provides a representation \eqref{eqpnkg}, with $\V{Q}$ given by \eqref{eqpnkgsQ}, for the Nevanlinna reproducing kernel $K(z,w)$ of $\mathfrak B$. In addition, by Theorem~\ref{tchucsg} there is a canonical subspace  $\mathfrak C$ such that  ${\mathfrak B} = W {\mathfrak C}$ for some $d \times d$ matrix polynomial $W(z)$ with ${\det}W(z)\not \equiv 0$. The proof of Theorem~\ref{tchpnk} and Lemma~\ref{tief} show  that the multiset of the Forney  indices of ${\mathbf P}(z)$  coincides with the multiset of the degrees of ${\mathfrak C}$.  This implies that the Forney indices are independent of the Nevanlinna representation \eqref{eqpnkg} of the kernel $K(z,w)$. In the special case when the defect numbers of $S_\mathfrak B$ are equal to $d$ this fact can also be proved directly by using \cite[Theorem~1.3]{gs}. In view of Remark~\ref{remdefect}, this remark substantiates the observations about the Forney indices and the defect numbers after Theorem~\ref{tchucsg} in the Introduction.
\end{remark}

\section{$Q$-functions}\label{s5}
\subsection{}
Let $M(z)$ be a generalized Nevanlinna $d \times d$ matrix function and denote by $\mathcal L(M)$ the reproducing kernel Pontryagin space with reproducing kernel $K_M(z,w)=K_{M, I_d}(z,w)$. By \cite[Theorem 2.1]{dlls}, the operator $S$ in $\mathcal L(M)$ of multiplication by the independent variable is a simple symmetric operator with equal defect numbers and its adjoint is given by
 \begin{align*}
 S^* & =
\overline{\lspan}\big\{\{K_M(\,\cdot\,,w^*) x, wK_M(\,\cdot\,,w^*) x\}\, :\, x \in
\mathbb C ^d, w \in {\hol}(M)\big\}\\[2mm]
& =\big\{\{f,g\} \in {\mathcal
L}(M)^2\,:\, \exists\, x,y \in \mathbb C ^d \ \text{such that}\ g(z)-z
f(z)\equiv x -M(z)y\big\} .
\end{align*}
It follows that for all $w \in {\hol}(M)$
$$
\ker(S^*-w) = \bigl\{ K_M(\,\cdot\,,w^*)x\,:\,x \in \mathbb C^d \bigr\} = \ran  E_w^*,
$$
where $E_w$ is considered as a mapping $E_w:\mathcal L(M) \rightarrow \mathbb C^d$.  Taking orthogonal complements we see that
 $$
\ran (S-\alpha) = {\mathcal L}(M) \cap \ker E_\alpha, \quad \alpha \in {\hol}(M).
 $$
Thus (\ref{tchpnkca}) and (\ref{tchpnkcb}) of Theorem \ref{tchpnk} hold.  Moreover,  \cite[Theorem 2.1]{dlls} and its proof imply that there is a constant invertible $d \times d$ matrix $T$ such that
 $$
T M(z)T^ *= M_0 + \begin{bmatrix} \widehat M(z) & 0 \\0 & 0 \end{bmatrix},
 $$
where $M_0$ is a constant self-adjoint $d \times d$ matrix and, if the defect numbers of $S$ are denoted by $l$, $\widehat M(z)$ is a generalized Nevanlinna $l \times l$ matrix function which is a $Q$-function for $S$.
The theorem below concerns a converse implication. But first we recall the notion of a $Q$-function.

\medskip

Let $S$ be a simple symmetric operator in a Pontryagin space $\cK$ with defect numbers equal to $l$. Let $A$ be a self-adjoint extension of $S$ in $\cK$ with a nonempty resolvent set $\rho(A)$. Let $\mu \in \rho(A)\setminus \mathbb R$ and define a function $\Gamma_\mu:\mathbb C^l \rightarrow \cK$ such that it is a linear bijection from $\mathbb C^l$ onto ${\ker}(S^*-\mu)$. Finally, for $z \in \rho(A)$ define the {\it defect mappings} $\Gamma_z:\mathbb C^l \rightarrow \cK$ by
$$
\Gamma_z=\left(I+(z-\mu)(A-z)^{-1}\right)\Gamma_\mu, \quad z \in \rho(A).
$$
Then $\Gamma_z$ is a bijection from $\mathbb C^l$ onto
${\ker}(S^*-z)$,
\begin{equation}\label{closedspan}
\cK=\overline{\lspan}\{\Gamma_z \mathbf c \,:\,
z \in \rho(A)\cap (\mathbb C\setminus \mathbb R), \mathbf c \in \mathbb C^l\}
\end{equation}
and,  by the resolvent identity, $\Gamma_w^*\Gamma_z=\Gamma_{z^*}^*\Gamma_{w^*}$, $w,z \in \rho(A)$. A {\it $Q$-function for} $S$ is by definition an $l \times l$ matrix function that satisfies the equation
\begin{equation}\label{Qkernel}
\dfrac{Q(z)-Q(w)^*}{z-w^*}=\Gamma_w^*\Gamma_z, \quad z,w \in \rho(A).
\end{equation}
Clearly, $Q(z)$ depends on the choice of the pair $\{A,\Gamma_z\}$ and if this choice has to be mentioned explicitly we shall say that $Q(z)$ is a $Q$-function for $S$ associated with the pair $\{A,\Gamma_z\}$. $Q(z)$ is uniquely determined up to an additive constant self-adjoint  $d \times d$ matrix $Q_0$:
$$Q(z)=Q_0-i{\im}\mu\,\Gamma^*_\mu\Gamma_\mu+
(z-\mu^*)\Gamma_\mu^*\Gamma_z, \quad Q_0=Q_0^*.
$$
From \eqref{closedspan} and the defining relation \eqref{Qkernel}
it follows that $Q(z)$ is a generalized Nevanlinna $l \times l$ matrix function with $\kappa$ negative squares where $\kappa$ is the negative index of the Pontryagin space $\cK$; in particular $Q(z)^*=Q(z^*)$. $Q$-functions in an indefinite setting were introduced and studied by M.G.~Krein and H.~Langer in \cite{kl1971, kl1973}.

 \medskip

\subsection{} The following theorem shows that the Nevanlinna pair $\{M(z), N(z)\}$ of matrix polynomials $M(z)$ and $N(z)$ in Theorem~\ref{tchpnk} can be chosen such that ${\det}N(z)\not \equiv 0$ and such that $N(z)^{-1}M(z)$ is essentially the $Q$-function for $S_\mathfrak B$. As before, by $\mathcal L(Q)$ we denote the reproducing kernel space with reproducing kernel given by \eqref{Qkernel}.

\begin{theorem}\label{tQ}
Let $\mathfrak B$ be a finite dimensional Pontryagin subspace of $\nC^d[z]$ for which the conditions \eqref{tchpnkca} and \eqref{tchpnkcb} of Theorem~{\rm~\ref{tchpnk}} hold. Denote by $l \in \{1,\ldots d\}$ the equal defect numbers of the symmetric operator $S_\mathfrak B$. Let $Q(z)$ be an $l \times l$ matrix $Q$-function for $S_\mathfrak B$. Then there is a $d \times d$ matrix polynomial $N(z)$ with ${\det}N(z)\not \equiv 0$ such that
$M(z)= N(z){\diag}( Q(z),  0)$ is a $d \times d$ matrix polynomial and
$\mathfrak B=N\left( \mathcal L(Q)\oplus \{0\}\right)$. In particular, $\{M(z),N(z)\}$ is a Nevanlinna pair of  matrix polynomials and $K_{M,N}(z,w)$ is the reproducing kernel of $\mathfrak B$.
\end{theorem}
\begin{proof}
Assume \eqref{tchpnkca} and \eqref{tchpnkcb} of Theorem~\ref{tchpnk}.  By Theorem~\ref{tchucsg} there is a $d\times d$ matrix function $W(z)$ with ${\det}W(z)\not \equiv 0$ such  that $\mathfrak B=W \mathfrak C$, where $\mathfrak C$ is a canonical subspace of $\nC^d[z]$. By Remark~\ref{remdefect} the defect numbers of the symmetric operator $S_\mathfrak B$ are both equal to  $l$ with $l \leq d$. We consider two cases: $l=d$ and $l<d$.

\medskip

\noindent
{\bf (i)} $l=d$. Let $Q(z)$ be the $Q$-function for $S_\mathfrak B$ associated with the pair $\{A, \Gamma_z\}$, where $A$ is a self-adjoint extension of $S_\mathfrak B$ and the defect mappings $\Gamma_z$ are defined above with $l=d$. Since $S_\mathfrak B$ is simple, the mapping
 $$
f \mapsto g\ \text{with}\ f(z) = \Gamma_{z^*}^*\Gamma_{w^*}\,x, \quad g(z)=  \left(\Gamma_{w^*}x\right)(z) , \qquad x \in \mathbb C^d,\ w \in \rho(A),
 $$
can be extended by linearity to a unitary mapping $U$ from $\mathcal L(Q)$
onto $\mathfrak B$. That $U$ is isometric follows from
$$
\bigl[ \Gamma_w^*\Gamma_z\,x,
\Gamma_v^*\Gamma_z\,y \bigr]_{\mathcal L(Q)}
=y^*\Gamma_w^*\Gamma_v\,x=
y^*\Gamma_{v^*}^*\Gamma_{w^*}\,x
=[ \Gamma_{w^*}x, \Gamma_{v^*}y]_{\mathfrak B},
\quad x ,y  \in \mathbb C^d.
$$

We claim that $U$ is the operator of multiplication by a $d \times d$ matrix function. To prove the claim we use the equality $\mathfrak B=W \mathfrak C$.  Since the defect numbers of $S_{\mathfrak B}$ are equal to $d$, the degrees of $\mathfrak C$ are all $\geq 1$ (see Remark~\ref{remdefect}) and hence the $d$ columns of $W(z)$ belong to $\mathfrak B$ and are linearly independent over $\mathbb C$. We denote by $\Gamma_{z^*}^*W$ the $d \times d$ matrix function defined by
 $$
(\Gamma_{z^*}^*W)x=\Gamma_{z^*}^*(W x), \quad x \in \mathbb C^d, \ z \in \rho(A).
 $$
We show that its inverse exists for $z \in \Omega:=\rho(A)\cap \{z \in \mathbb C\,:\, {\det}W(z)\neq 0\}$. Suppose there is an $x \in \mathbb C^d$ such that
$\left(\Gamma_{z^*}^*W\right)  x=0$. Then for all $ y\in \mathbb C^d$
 $$
 0= \bigl[ \left(\Gamma_{z^*}^*W\right)  x, y \bigr]_{\mathbb C^d} =
 \bigl[ \Gamma_{z^*}^* (W x) ,  y \bigr]_{\mathbb C^d}=
 \bigl[W x, \Gamma_{z^*} y \bigr]_{\mathfrak B},
 $$
hence $W  x \in \ker(S^*_{\mathfrak B}-z^*)^\perp={\ran}(S_{\mathfrak B}-z)={\mathfrak B}\cap E_z$. That is, $W(z) x=0$ and it follows that $x=0$. This proves that $(\Gamma_{z^*}^*W)^{-1}$ is well defined for all $z \in \Omega$.
We set $N(z)=W(z)(\Gamma_{z^*}^*W)^{-1}$. Clearly, ${\det}N(z)\not \equiv 0$. We have shown
that $U$ coincides with multiplication by $N(z)$ if we have proved that
$$
N(z)\Gamma_{z^*}^*\Gamma_{w^*}x=
\left(\Gamma_{w^*}x\right)(z),\qquad x \in \mathbb C^d,\ w \in \rho(A),\ z \in \Omega,
$$
or, equivalently, that with $y(z,w, x):= W(z)^{-1}\left(\Gamma_{w^*}x\right)(z) \in \mathbb C^d$
$$
\Gamma_{z^*}^*\Gamma_{w^*}x
=\Gamma_{z^*}^*\left(W y(z,w,x)\right),
\qquad x \in \mathbb C^d,\ w \in \rho(A),\
z \in \Omega.
$$
But this equality holds, since $\Gamma_{z^*}^*\left({\ran}( S_{\mathfrak B}-z)\right)=\{0\}$ and
$$
\Gamma_{w^*}x-W y(z,w, x) \in {\mathfrak B} \cap E_z =
{\ran}(S_{\mathfrak B}-z).
$$
This completes the proof of the claim that $U$ is multiplication by $N(z)$.
It follows from \cite[Theorem 1.5.7]{adrs} and its proof that the formula for the kernel $K(z,w)$ of $\mathfrak B$ is given by
$$K(z,w)=N(z)\dfrac{Q(z)-Q(w)^*}{z-w^*}N(w)^*$$
and hence $\mathfrak B=N\mathcal L(Q)$.

It remains to show that $M(z)=N(z)Q(z)$ and $N(z)$ are matrix polynomials. Since the elements of the space $\mathfrak B$ are polynomials, the matrix function $z \mapsto K(z,w)$ is a matrix polynomial, hence the matrix function
 $$
M(z)-N(z)Q(w)^*=N(z)Q(z)-N(z)Q(w)^*
=(z-w^*)K(z,w)N(w)^{-*}
 $$
is a matrix polynomial in $z$. Thus if $N(z)$ is a matrix polynomial, then so is $M(z)$.
It remains to show that $N(z)$ is a polynomial. For this we note that the above formula implies that for $x \in \mathbb C^d$
 $$
N(z)\dfrac{Q(\mu^*)-Q(w)^*}{\mu^*-w^*}x=
\dfrac{(z-\mu^*)K(z,\mu)N(\mu)^{-*}-(z-w^*)
K(z,w)N(w)^{-*}}{\mu^*-w^*} \, x.
 $$
The right hand side is a matrix polynomial in $z$ and hence it follows from the equality that  $N(z)$ is a matrix polynomial if we can show that
 $$
{\mathbb C}^d
 = {\lspan} \bigl\{\Gamma_{\mu}^*\Gamma_{w^*}x\,:\, w\in \rho(A)\cap ({\mathbb C}\setminus {\mathbb R}),\ x \in {\mathbb C}^d \bigr\}.
 $$
To prove this equality we argue by contradiction and suppose it is not true. Then there is a nonzero vector $x \in {\mathbb C}^d$ orthogonal to the set on the right hand side, that is,
 $$
 \bigl[ \Gamma_{w^*}y, \Gamma_{\mu} x \bigr]_{\mathfrak B}=0, \quad
w\in \rho(A)\cap ({\mathbb C}\setminus {\mathbb R}),\ y \in {\mathbb C}^d.
 $$
Since $S_{\mathfrak B}$ is simple and $\Gamma_\mu$ is injective, we find that $\Gamma_{\mu} x=0$
and that $x=0$, which contradicts the choice of the nonzero vector $x$.

\medskip

\noindent {\bf (ii)} $l <d$. Then $\mathfrak C=\mathfrak C_1\oplus \{0\}$, where $\mathfrak C_1$ is a canonical subspace of $C^{l}[z]$
of which the degrees are all $\geq 1$. Using the relation $\mathfrak B=W\left( {\mathfrak C}_1\oplus \{0\} \right)$ we equip $\mathfrak C_1$ with an indefinite inner product that makes $W$ an isomorphism. Then $S_\mathfrak B$ and $S_{\mathfrak C_1}$ are isomorphic: $S_{\mathfrak C_1}=WS_{\mathfrak B}  W^{-1}$, hence $S_{\mathfrak C_1}$ is symmetric and has defect numbers equal to  $l$. Thus \eqref{tchpnkca} holds and it is not difficult to verify that also \eqref{tchpnkcb}
holds on $\mathfrak C_1$. Finally, since $Q(z)$ is the
$Q$-function for $S_\mathfrak B$ associated with the pair $\{A, \Gamma_z\}$, $Q(z)$ is the $Q$-function for $S_{\mathfrak C_1}$ associated with the pair $\{W^{-1}AW, W^{-1}\Gamma_z\}$. This all shows that we may apply  part (i) of this proof (with $W(z)=I_l$): There exists an $l \times l$  matrix polynomial $N_1(z)$ with ${\det}N_1(z) \not \equiv 0$ such that
$N_1(z)Q(z)$ is an $l \times l$  matrix polynomial and
$\mathfrak C_1=N_1 \mathcal L(Q)$. It follows that if
$$N(z)=W(z)\,{\diag}(N_1(z), I_{d-l}),$$
then ${\det}N(z)\not \equiv 0$, $N(z)\,{\diag}(Q(z), 0)$ is a $d \times d$ matrix polynomial and
$\mathfrak B=N\left( \mathcal L(Q) \oplus \{0\}\right)$.
\end{proof}

\section{Corollaries and examples}\label{s6}

In the next corollary we extend Theorem~\ref{tchpnk} to finite dimensional Pontryagin spaces of rational vector functions. A rational Nevanlinna kernel is a kernel of the form $K_{M, N}(z,w)$ as in \eqref{eqNevMNG}, in which $M(z)$ and $N(z)$ are rational $d \times d$ matrix functions satisfying \eqref{iPneuG} and \eqref{irankzG}.

\begin{corollary}\label{corrat}
Let $\fB$ be a finite dimensional Pontryagin space of rational $d \times 1$ vector functions and let $\Omega \subset \nC$ be the finite set of all the poles of the functions in $\fB$. Denote by $S_\fB$ the operator of multiplication by the independent variable in $\fB$ and by $E_\alpha$ the operator of evaluation at a point $\alpha \in \nC$. Then the reproducing kernel of $\fB$ is a rational Nevanlinna kernel if and only if the following two conditions hold:
\begin{enumerate}[{\rm (a)}]
\item \label{corratia}
The operator $S_\fB$ is symmetric in $\fB$.
\item \label{corratib}
For some $\alpha \in \nC\setminus\Omega$ we have  $\ran \bigl(S_\fB -\alpha \bigr) = \fB \cap \ker E_{\alpha}$.
\end{enumerate}
\end{corollary}

\begin{proof}
Assume (\ref{corratia}) and (\ref{corratib}). Let $q(z)$ be the monic scalar polynomial of minimal degree such that $\mathfrak B':=\{q(z)f(z)\,:\, f \in  {\mathfrak B}\}$ consists of polynomials. Equip $\mathfrak B'$ with the Pontryagin space inner product that makes the mapping $q: \mathfrak B \rightarrow \mathfrak B'$ of multiplication by $q(z)$ a unitary mapping. Then items \eqref{tchpnkca} and \eqref{tchpnkcb} of Theorem~\ref{tchpnk} hold for $\mathfrak B'$. Hence $\mathfrak B'$ has a polynomial reproducing Nevanlinna kernel $K_{M,N}(z,w)$. It follows that
$\mathfrak B$ has reproducing kernel $K_{M/q, N/q}(z,w)$.

Now assume $K_{M,N}(z,w)$ is a rational reproducing Nevanlinna kernel of $\mathfrak B$. Let $r(z)$ be a polynomial such that $r(z)M(z)$ and
$r(z)N(z)$ are polynomials and hence form a polynomial Nevanlinna pair
$\{r(z)M(z), r(z)N(z)\}$. Then $K_{rM,rN}(z,w)$ is a polynomial reproducing Nevanlinna kernel of the space $\mathfrak B '':=\{r(z)f(z)\,:\,f(z) \in \mathfrak B\}$ equipped with the inner product that makes multiplication by $r(z)$ an isomorphism from $\mathfrak B$ onto $\mathfrak B''$. Since the elements of $\mathfrak B''$ are polynomials, we can apply Theorem~\ref{tchpnk} to conclude that
items \eqref{tchpnkca} and \eqref{tchpnkcb} hold for the space $\mathfrak B''$. Since multiplication by $r(z)$ and by  $z$ commute,
\eqref{tchpnkca} implies (\ref{corratia}).  By Theorem~\ref{tchucsg} and \eqref{eqdetrannul}, the equality
\begin{equation} \label{b"}
\ran \bigl(S_{\mf B''} - \alpha \bigr) = \mf B'' \cap \ker E_{\alpha}
\end{equation}
holds for all but finitely many $\alpha \in \mathbb C$. Choose
$\alpha \in \nC\setminus\Omega$ such that \eqref{b"} is valid. Then for this $\alpha$ item (\ref{corratib}) holds.
\end{proof}

\begin{corollary} \label{Cor2}
Let $(\fB, [\,\cdot\,,\,\cdot\,]_\fB)$ be a finite dimensional Pontryagin subspace of $\mathbb C^{d}[z]$ whose reproducing kernel is a Nevanlinna kernel determined by a generalized Nevanlinna pair. Let $J$ be a fundamental symmetry on $\mathfrak B$. Then the Hilbert space $(\fB, [J\,\cdot\,,\,\cdot\,]_\fB)$
has a reproducing Nevanlinna kernel determined by a Nevanlinna pair if and only if $S_\mathfrak B$ is symmetric in this space.
\end{corollary}
The corollary follows from Theorem~\ref{tchpnk}, because condition \eqref{tchpnkcb} is independent of the topology on $\mathfrak B$.

\begin{example}\label{exci} Consider the subspace $\mathfrak B$ of $\mathbb C^2[z]$ spanned by the columns of the matrix
$$B(z)=\begin{bmatrix} 1&z&z^2&0 \\ 0&0&0&1
\end{bmatrix}$$
and equipped with the inner product $[\,\cdot\,,\,\cdot\,]_\fB$ so that
$$G=\begin{bmatrix} 1&0&0&0\\ 0&0&0&1 \\ 0&0&1&0\\0&1&0&0\end{bmatrix}$$
is the Gram matrix associated with $B(z)$: $G=[B,B]_\fB$. The spectral
decomposition of $G$ is $G=UJU^*$ with unitary matrix
\[
U = \frac{1}{\sqrt{2}} \left[\!\!\!
\begin{array}{rrrr}
 0 & 0 & 0 & \sqrt{2} \\
 - 1 & 1 & 0 & 0 \\
 0 & 0 & \sqrt{2} & 0 \\
 1 &  1 & 0 & 0
\end{array}
\right] \quad \text{and} \quad J=\left[\!\!\!
\begin{array}{rrrr}
 -1 & 0 & 0 & 0 \\
 0 & 1 & 0 & 0 \\
 0 & 0 & 1 & 0 \\
 0 & 0 & 0 & 1
\end{array}
\right].
\]
It follows that $(\mathfrak B, [\,\cdot\,,\,\cdot\,]_\fB)$ is a Pontryagin space with positive index $3$ and negative index $1$.
The equality
$[BU, BU]_\fB=J$
defines a fundamental decomposition of $\mathfrak B$ with corresponding fundamental symmetry $\mathcal J$ determined by
$\mathcal J BU= BUJ$.
In the Hilbert space inner product $[\,\cdot\,,\,\cdot\,]_{\mathcal J}:=[{\mathcal J}\,\cdot\,,\,\cdot\,]_\fB$ we have
$[BU, BU]_{\mathcal J}=J^2=I_n$
and hence
$[B,B]_{\mathcal J}=I_n$. The operator $S_\mathfrak B$
is symmetric in the Pontryagin space $(\mathfrak B, [\,\cdot\,,\,\cdot\,]_\fB)$, but not in the Hilbert space $(\mathfrak B, [\,\cdot\,,\,\cdot\,]_{\mathcal J})$. Since $\mathfrak B$ is
a canonical subspace of $\mathbb C^2[z]$, Theorem \ref{tchucsg} implies that Theorem \ref{tchpnk} \eqref{tchpnkcb} holds in $\mathfrak B$. Hence, according to Theorem \ref{tchpnk}, the Pontryagin space
$(\mathfrak B, [\,\cdot\,,\,\cdot\,]_\fB)$ has a reproducing Nevanlinna kernel, whereas the reproducing Hilbert  space $(\mathfrak B, [\,\cdot\,,\,\cdot\,]_{\mathcal J})$ does not have a reproducing Nevanlinna kernel. \hfill $\square$
\end{example}

\begin{corollary} \label{Cor3}
Let $\fB$ be a finite dimensional Pontryagin subspace of $\mathbb C^{d}[z]$ whose reproducing kernel is a Nevanlinna kernel. Let $\fB_0$ be a Pontryagin subspace of $\mathfrak B$. Then the reproducing kernel of $\fB_0$ is a Nevanlinna kernel if and only if
for some $\alpha \in \nC$ we have  $\ran \bigl(S_{\fB_0} -\alpha \bigr) = \fB_0 \cap \ker E_{\alpha}$.
\end{corollary}
The corollary follows from Theorem~\ref{tchpnk}, because the hypothesis implies that $S_{\fB_0}$, being a subset of $S_\fB$,  is symmetric in $\mathfrak B_0$, that is, that  \eqref{tchpnkca} holds
for $S_{\mathfrak B_0}$.
\begin{example}
Consider the Hilbert subspace
\[
{\mathfrak B}_0 = \lspan \left\{ \begin{bmatrix} 1 \\ 0\end{bmatrix}, \begin{bmatrix} z^2 \\ 0 \end{bmatrix} \right\}
\]
of the space ${\mathfrak B}$ in Example~\ref{exci}. Then for arbitrary $\alpha \in \nC$ we have
\[
\ran \bigl(S_{\fB_0} - \alpha \bigr) = \{0\} \quad \text{and} \quad
 \fB_0 \cap \ker E_{\alpha} =  \lspan \left\{ \begin{bmatrix} z^2 -\alpha^2 \\ 0 \end{bmatrix} \right\}.
\]
Thus the condition $\ran \bigl(S_{\fB_0} -\alpha \bigr) = \fB_0 \cap \ker E_{\alpha}$ does not hold for any $\alpha \in \nC$. Hence   Corollary~\ref{Cor3} implies that the reproducing kernel of $\fB_0$, which is calculated to be
\[
K(z,w) = \begin{bmatrix} 1 + z^2 w^{*2} & 0 \\[6pt] 0 & 0
\end{bmatrix}, \quad z, w \in \nC,
\]
is not a Nevanlinna kernel. This fact can be verified using \cite[Theorem~1.3]{gs}. First observe that for all $z, w \in \nC$ we have
\[
(z - w^*) K(z,w) = M(z) N(w)^*- N(z) M(w)^* = \bigl[ M(z)  \ \ N(z) \bigr] \left[\!\! \begin{array}{r}
N(w)^* \\[6pt] -M(w)^* \end{array}\!\! \right],
\]
where
\[
N(z) = \begin{bmatrix} 1 & z^2 \\ 0 & 0
\end{bmatrix} \quad \text{and} \quad M(z) = z N(z), \quad z \in \nC.
\]
Now \cite[Theorem~1.3 and Section~4]{gs} imply that for any $2\times 2$ matrix polynomials $M_1(z)$ and $N_1(z)$ such that
\begin{equation} \label{eqnnev}
(z - w^*) K(z,w) =  \bigl[ M_1(z)  \ \ N_1(z) \bigr] \left[\!\! \begin{array}{r}
N_1(w)^* \\[6pt] - M_1(w)^* \end{array}\!\! \right], \quad z,w \in \nC,
\end{equation}
there exists a $4\times 4$ invertible matrix $S$ such that
\[
\bigl[ M_1(z)  \ \ N_1(z) \bigr] = \bigl[ M(z)  \ \ N(z) \bigr] S, \quad z \in \nC.
\]
Hence \eqref{eqnnev} yields that $\rank \bigl[ M_1(z)  \ \ N_1(z) \bigr] = 1$ for all $z\in\nC$. Consequently $K(z,w)$ is not a Nevanlinna kernel.
Using the same results from \cite{gs} one can also show that the
scalar reproducing kernel $K(z,w)=1+z^2w^{*2}$ of the Hilbert space with orthonormal basis $\{1,z^2\}$  is not a Nevanlinna kernel.
 \hfill $\Box$
\end{example}

\medskip

We end the paper with two examples in which $\det K(z,w) \equiv 0$. These examples also show that the proof of Theorem \ref{tQ} is constructive.
\begin{example}\label{example1}
Consider the space $\mathfrak B$ with reproducing kernel
\begin{equation*}
\mathcal K(z,w) =\begin{bmatrix}  0 & 0 & -1 \\ 0 & 0 & -w^* \\
-1 & -z & 0
\end{bmatrix}.
\end{equation*}
 We show that, even though ${\det}K(z,w) \equiv 0$, the kernel is a Nevanlinna kernel. We follow the proof of the first part of Theorem~\ref{tQ} and construct two Nevanlinna pairs that determine $K(z,w)$.
The space $\mathfrak B$ is spanned by the columns of the $3 \times 4$ matrix polynomial
$$B(z)=\begin{bmatrix} 1& 0& 0&0
\\0 &1 & 0 &0 \\ 0& 0&1&z  \end{bmatrix}.$$
It follows that
$\mathfrak B$ is a canonical subspace of $\mathbb C^3[z]$ with degrees $1$, $1$ and $2$.
The Gram matrix associated with $B(z)$ is given by
$$G= [  B,  B]_\fB=\begin{bmatrix} 0 & -I\\ -I & 0 \end{bmatrix},
$$
hence $\mathfrak B$ is a Pontryagin space with positive and negative index $2$. The operator of multiplication by $z$ on $\mathfrak B$ is given by
$$
S_\mathfrak B=\left\{\left\{  B\begin{bmatrix} 0 \\ 0 \\ a \\ 0 \end{bmatrix}
,  B \begin{bmatrix} 0 \\ 0 \\ 0 \\ a
\end{bmatrix}\right\} \,:\,  a \in \mathbb C \right\}.
$$
It is easy to  see that \eqref{tchpnkca} and \eqref{tchpnkcb} of Theorem~\ref{tchpnk} are satisfied. The defect numbers of $S_\mathfrak B$ are both equal to $3$, see Remark~\ref{remdefect}. The $Q$-function of $S_\mathfrak B$ associated with the self-adjoint extension
$$
A=\left\{\left\{  B\begin{bmatrix} a \\ 0 \\ b \\ 0 \end{bmatrix}
,  B \begin{bmatrix} 0 \\ c \\ 0 \\ d
\end{bmatrix}\right\} \,:\,  a,b,c,d \in \mathbb C \right\}
$$
of $S_\mathfrak B$ (which is multi-valued and has a nonempty resolvent set) and the defect mappings
$\Gamma_z : \mathbb C^3 \rightarrow \ker(S^*-z)$ defined by
$$\Gamma_z =\left(I+(z-\iu)(A-z)^{-1}\right)\Gamma_\iu= B
\begin{bmatrix}\iu /z& 0&0 \\ \iu & 0&0 \\ 0&\iu/z& 0 \\ 0&0&1
\end{bmatrix}
$$
is
$$
Q(z)=Q_0+ \begin{bmatrix}
0& 1/z& \iu z \\ 1/z& 0&0 \\-\iu z& 0&0
\end{bmatrix}, \quad Q_0=Q_0^*.
$$
We find that $X(z)=Y(z)Q(z)$ and $Y(z)=\left(\Gamma_{z^*}^*\mathbf I\right)^{-1}$ form a full generalized Nevanlinna pair of matrix polynomials:
$$
X(z)=Y(z)Q_0+
\begin{bmatrix}-\iu &0&0 \\ \iu z&0&0\\ 0&-\iu &z^2
\end{bmatrix}
\quad \text{and} \quad
Y(z)=
\begin{bmatrix} 0&-\iu z&0 \\ 0&0&-1 \\ -\iu z&0&0
\end{bmatrix}
$$
such that $K(z,w)=K_{X,Y}(z,w)$ and $\bigl[ X(z) \,\ Y(z)\bigr]$ is row reduced with Forney indices $1$, $1$ and $2$, which is in accordance with Remark~\ref{rlast}.

The $Q$-function associated with the self-adjoint operator extension $A$ of $S_\mathfrak B$ and defect mappings $\Gamma_z$ defined by
$$
A B= B
\begin{bmatrix} 0&1&0&0\\0&0&0&0\\0&0&0&0\\0&0&1&0
\end{bmatrix},\quad
\Gamma_z=B
\begin{bmatrix} -1/z^2& 0&0 \\ -1/z& 0&0 \\
0& \iu /z& 0 \\ 0& -(z-\iu)/z^2& \iu /z
\end{bmatrix}
$$
is given by
$$
Q(z)=Q_1+\begin{bmatrix}  0 & \dfrac{-\iu z^2 +z-\iu}{z^2} & \dfrac{-\iu}{z} \\[2mm] \dfrac{\iu z^2+z+\iu}{z^2} & 0 &0 \\[2mm] \dfrac{\iu}{z} & 0 & 0
\end{bmatrix}, \quad Q_1=Q_1^*.
$$
Again we find that $M(z)=N(z)Q(z)$ and $N(z)=\left(\Gamma_{z^*}^*\mathbf I\right)^{-1}$ are matrix polynomials:
$$M(z)=N(z)Q_1+ \begin{bmatrix} z &0&0 \\ 1&0&0 \\
0& -\iu z^2+z-\iu  & -\iu z \end{bmatrix}, \
N(z)=\begin{bmatrix} 0 &-\iu z & z+\iu \\ 0& 0 & -\iu z \\ z^2 & 0&0
\end{bmatrix}
$$
which form a full generalized Nevanlinna pair such that
$K(z,w)=K_{M,N}(z,w)$ and $\begin{bmatrix} M(z) & N(z) \end{bmatrix}$ is row reduced with Forney indices $1$, $1$ and $2$.
$\hfill \square$
\end{example}

Two generalized Nevanlinna pairs $\{X(z),Y(z)\}$ and $\{M(z),N(z)\}$ of $d\times d$ matrix polynomials define the same Nevanlinna kernel if one is a $J$-unitary transformation of the other, that is, if they are connected via the formulas
$$M(z)=X(z)A+Y(z)C, \quad N(z)=X(z)B+Y(z)D,$$
where $A$, $B$, $C$ and $D$ are constant $d \times d$ matrices such that if we set
$$U =\begin{bmatrix} A & B \\ C & D \end{bmatrix} \quad \text{and} \quad
J=\begin{bmatrix} 0 & \iu I_d \\ - \iu I_d & 0 \end{bmatrix},
$$
then $U$ is $J$-unitary: $UJU^*=J$. The pairs $\{X(z), Y(z)\}$ and $\{M(z), N(z)\}$ in Example \ref{example1} (with arbitrary constant self-adjoint  matrices $Q_0$ and $Q_1$)  are connected via a $J$-unitary transformation. The following example shows that the converse of the foregoing statement does not hold.

\begin{example}\label{example2}
Consider the Nevanlinnna pair $\{X(z), Y(z)\}$ given by
$$X(z)={\diag}(z,0,z^2), \quad Y(z)={\diag}(0,z,z).$$
Then the space $\mathfrak B$ with reproducing kernel $K_{X,Y}(z,w)={\diag}(0,0,zw^*)$ is a $1$-dimensional Hilbert space: it is spanned by
$
B(z)=\begin{bmatrix} 0 & 0 & z \end{bmatrix}^\top
$
and the corresponding Gram matrix is $G=[B,B]_\fB=1$.
Note that ${\det}X(z)\equiv 0$, ${\det}Y(z)\equiv 0$ and $\mathbf P(z):=\bigl[ X(z) \,\  Y(z)\bigr]$ does not have full rank at $z=0$: $\mathbf P(0)=0$. We show that the pair $\{X(z), Y(z)\}$ can be replaced by a Nevanlinna pair $\{M(z), N(z)\}$ such that ${\det}N(z)\not \equiv 0$.

Since $\mathbf P(z)$ does not have full rank for all $z \in \mathbb C$, to calculate the Forney indices we must first apply Lemma~\ref{tief}. We write $\mathbf P(z)$ as $\mathbf P(z)=G(z)\mathbf T(z)$ with
 $$
G(z)=\begin{bmatrix}0&0&z\\ 0&z&0\\z&0&0 \end{bmatrix},\quad  \mathbf T(z)=\begin{bmatrix}0&0&z&0&0&1 \\ 0&0&0&0&1&0 \\ 1&0&0&0&0&0 \end{bmatrix}.
 $$
Since $\mathbf T(z)$  has full rank for all $z \in \mathbb C$ and is row reduced, the Forney indices of ${\mathbf P}(z)$ are those of ${\mathbf T}(z)$ and they are $\mu_1=1$, $\mu_2=\mu_3=0$. This fits in well with the observations after Theorem~\ref{tchucsg} indicating that $\dim {\mathfrak B}=1$ and the defect numbers of the symmetric operator $S_{\mathfrak B} = \{\{0,0\}\}$ are both equal to $1$. We follow part (ii) of the proof of Theorem~\ref{tQ} and write ${\mathfrak B} = W\left({\mathfrak C}_1 \oplus \{0\}\right)$ with
$$
W(z)=\begin{bmatrix} 0&0&1\\ 0&1&0 \\ z&0&0 \end{bmatrix}
$$
and $\mathfrak C_1=\mathbb C$. We make the multiplication operator $W$ an isometry when $\mathbb C$ is equipped with the Euclidean inner product. Then $S_{\mathfrak C_1}=\{\{0,0\}\}$ is symmetric. The defect subspaces ${\ker}(S_{{\mathfrak C}_1}^*-z)$ all coincide with $\mathbb C$ and $A$ is a self-adjoint extension of $S_{{\mathfrak C}_1}$ if and only if $A=A_m$, the operator of multiplication by $m$, $m \in \mathbb  R$, or $A=A_{\rm rel}=\bigl\{\{0,c\} \,:\, c \in \mathbb C\bigr\}$. Since ${\mathfrak C}_1$ is a Hilbert space, all self-adjoint operators and relations have a non-empty resolvent set.
Choose $\mu \in \mathbb C \setminus \mathbb R$, $\gamma \in \mathbb C \setminus \{0\}$ and define $\Gamma_\mu: \mathbb C \rightarrow {\ker}(S_{\mathfrak C_1}^*-\mu)=\mathbb C$ by $\Gamma_\mu x=\gamma x$, $x \in \mathbb C$. Then the $Q$-function $q(z)$ of $S_{\mathfrak C _1}$ associated with $\{A, \Gamma_z\}$ is given by
$$
q(z)=q_0+\left\{\begin{array}{ll}
\dfrac{|m-\mu|^2|\gamma|^2}{m-z} & \text{if}\ A=A_m, \\[14pt]
|\gamma|^2 z & \text{if}\ A=A_{\rm rel},
\end{array}\right.
$$
where $q_0$ is an arbitrary real number, and
$$
\left(\Gamma_{z^*}^*1\right)^{-1}=\left\{
\begin{array}{ll} \dfrac{m-z}{\gamma^*(m-\mu^*)}& \text{if}\ A=A_m, \\[14pt]
\dfrac{1}{\gamma^*}& \text{if}\ A=A_{\rm rel}.
\end{array}\right.
$$
We find that $K_{X,Y}(z,w)=K_{M,N}(z,w)$ with matrix polynomials
$$
M(z)=N(z) \begin{bmatrix} q(z)&0&0 \\ 0&0&0 \\ 0&0&0\end{bmatrix} \quad \text{and} \quad
N(z)=W(z)\begin{bmatrix} \left(\Gamma_{z^*}^*1\right)^{-1}& 0 & 0 \\ 0&1&0 \\ 0&0&1
\end{bmatrix}.
$$
It is easy to see that the Nevalinna pairs $\{X(z), Y(z)\}$ and $\{M(z), N(z)\}$ are not related via a $J$-unitary transformation.
$\hfill \square$
\end{example}


\begin{thebibliography}{99}
\bibitem{amanuscript} D.~Alpay, handwritten manuscript (French), unpublished.

\bibitem{a} D.~Alpay,
A structure theorem for reproducing kernel Pontryagin spaces, J.~Comput. Appl.~Math. 99  (1998) 413--422.

\bibitem{abook}D.~Alpay, The Schur algorithm, reproducing kernel spaces and system theory, SMF/AMS Texts and Monographs 5, Panorama et Synth\`eses 6, 1998.


\bibitem{abds1} D. Alpay, P. Bruinsma, A.~Dijksma, H.S.V.~de~Snoo,
Interpolation problems, extensions of symmetric operators and
reproducing kernel spaces I, Oper. Theory Adv. Appl. 50, Birkh\"{a}user, Basel, 1991, pp. 35--82.

\bibitem{abds2} D. Alpay, P. Bruinsma, A. Dijksma, H.S.V.~de~Snoo,
Interpolation problems, extensions of symmetric operators and
reproducing kernel spaces II, Integral Equations Operator Theory
14  (1991) 465--500;  missing Section 3, Integral Equations Operator Theory 15 (1991) \mbox{378--388}.

\bibitem{adrs} D. Alpay, A. Dijksma,  J. Rovnyak, H.S.V.~de~Snoo,
Schur functions, operator colligations, and reproducing
kernel Pontryagin spaces, Oper. Theory Adv. Appl. 96,
Birkh\"{a}user, Basel, 1997.

\bibitem{ad} D.~Alpay, H.~Dym, On applications of reproducing kernel spaces to the Schur algorithm and rational $J$ unitary factorization, Oper. Theory Adv. Appl. 18, Birkh\"{a}user, Basel, 1986, pp. 89--159.

\bibitem{ade} A.~Amirshadyan, V.~Derkach, Interpolation in generalized Nevanlinna and Stieltjes classes, J.~Operator Theory 42 (1999) 145--188.

\bibitem{acdJFA} T.Ya. Azizov, B. \'{C}urgus, A. Dijksma,
Standard symmetric operators in Pontryagin spaces: A generalized
von Neuman formula and minimality of boundary coefficients, J.~Funct. Anal. 198 (2003) 361--412.

\bibitem{acdNew} T.Ya. Azizov, B. \'{C}urgus, A. Dijksma,
Eigenvalue problems with boundary conditions
depending polynomially on the eigenparameter
 (tentative title), in preparation.


\bibitem{bdhs} J.~Behrndt, V.A.~Derkach, S.~Hassi, H.S.V.~de~Snoo, A realization theorem for generalized Nevanlinna families, to appear in Oper. Matrices.

\bibitem{dbbook} L. de Branges, Hilbert spaces of entire functions, Prentice Hall, Englewood Cliffs, 1968.

\bibitem{db1966} L.~de~Branges,
The expansion theorem for Hilbert spaces of entire functions, in:  Entire Functions and Related Parts of Analysis (Proc. Sympos. Pure Math., La Jolla, Calif., 1966), Amer. Math. Soc., Providence, 1968, pp. 79--148.


\bibitem{dbr} L.~de~Branges, J.~Rovnyak,
Canonical models in quantum scattering theory, in: Perturbation Theory and its Applications in Quantum Mechanics (Proc. Adv. Sem. Math. Res. Center, U.S. Army, Theoret. Chem. Inst., Univ. of Wisconsin, Madison, Wis., 1965) Wiley, New York, 1966, pp. 295--392.

\bibitem{cdr} B.~\'{C}urgus, A.~Dijksma, T.T.~Read, The linearization of boundary eigenvalue problems and reproducing kernel Hilbert spaces, Linear Algebra Appl. 329 (2001) \mbox{97--136}.

\bibitem{dlls} A.~Dijksma, H.~Langer, A.~Luger,
Y.~Shondin,
{Minimal realizations of scalar generalized Nevanlinna functions
related to their basic factorization}, Oper. Theory Adv. Appl. {154}, Birkh\"{a}user, Basel, 2004, pp. 69--90.

\bibitem{dls1984} A. Dijksma, H. Langer, H.S.V. de Snoo, {Selfadjoint
${\Pi}_{\kappa}$-extensions of symmetric subspaces: an abstract
approach to boundary problems with spectral parameter in the
boundary conditions}, Integral Equations Operator Theory {7} (1984) 459--515; {Addendum}, Integral Equations Operator Theory {7} (1984) 905.

\bibitem{dls1988} A. Dijksma, H. Langer, H.S.V. de Snoo, {Hamiltonian systems with eigenvalue depending boundary conditions}, Oper. Theory Adv. Appl. {35}, Birkh\"{a}user, Basel, 1988, pp. 37--83.

\bibitem{as} A.~Dijksma, H.S.V.~de~Snoo, {Symmetric and selfadjoint
relations in Krein spaces} I, Oper. Theory Adv. Appl. {24}, Birkh\"{a}user, Basel, 1987, pp. 145--166.

\bibitem{dv} H.~Dym, D.~Volok, {Zero distribution of matrix ploynomials}, Linear Algebra Appl. {425} (2007) 714--738.

\bibitem{F}
P.~Fuhrmann,  On duality in some problems of geometric control.  Acta
Appl. Math.  91  (2006) 207--251.

\bibitem{frg} F.R.~Gantmacher, {Matrizentheorie}, VEB Deutscher Verlag der Wissenschaften, Berlin, 1986.


\bibitem{gs} I.~Gohberg, T.~Shalom,
On Bezoutians of nonsquare matrix polynomials and inversion of matrices with nonsquare blocks, Linear Algebra Appl.
137-138 (1990) \mbox{249--323}.

\bibitem{ikl}
I.S.~Iohvidov, M.G.~Krein, H.~Langer,
{Introduction to the spectral theory of operators in spaces with
  an indefinite metric},
Akademie-Verlag, Berlin, 1982.

\bibitem{KTh}
T.~Kailath, {Linear systems}, Prentice-Hall,  Englewood Cliffs, 1980.

\bibitem{kw19991} M.~Kaltenb\"ack, H.~Woracek,
{Pontryagin spaces of entire functions} I, Integral Equations Operator Theory {33} (1999) 34--64.

\bibitem{kw2009} M.~Kaltenb\"ack, H.~Woracek,
{Pontryagin spaces of entire functions} V, Acta Sci. Math. (Szeged) 77 (2011) 179--292.

\bibitem{kl1971} M.G.~Krein, H.~Langer, { Defect subspaces and generalized resolvents of an Hermitian operator in the space $\Pi_\kappa$}, Funct. Anal. Appl. {5} (1971) 136--146; ibid {5} (1971) 217--228.

\bibitem{kl1973} M.G.~Krein, H.~Langer, {\"Uber die $Q$-Funktion eines $\pi$-hermiteschen Operators im Raume $\Pi_\kappa$}, Acta Sci. Math. (Szeged) {34} (1973) 191--230.

\bibitem{kl1977} M.G.~Krein, H.~Langer, {\"Uber einige Fortsetzungsprobleme, die eng mit der Theorie hermitescher Operatoren im Raume $\Pi_\kappa$ zusammenh\"angen. I. Einige Funktionenklassen und ihre Darstellungen}, Math. Nachr. {77} (1977) 187--236.

\bibitem{kl1981} M.G.~Krein, H.~Langer, {Some propositions on analytic matrix functions to the theory of operators in the space $\Pi_\kappa$}, Acta. Sci. Math. (Szeged) {43} (1981) 181--205.

\bibitem{L} H. Elton Lacey, {The Hamel dimension of any infinite dimensional separable Banach space is~$c$}, Amer. Math. Monthly {80} (3) (1973) 298.

\bibitem{McE} R.~J.~McEliece, {The algebraic theory of convolutional codes. Handbook of coding theory}, Vol. I, II,  North-Holland, Amsterdam, 1998, pp. 1065--1138.

\bibitem{XL1} Xian-Jin~Li, {The Riemann hypothesis for polynomials orthogonal on the unit circle}, Math. Nachr. {166} (1994) 229--258.

\bibitem{XL2} Xian-Jin~Li, {On Reproducing Kernel Hilbert Spaces of Polynomials}, Math. Nachr. {186} (1997) 115--148.

\end{thebibliography}
\end{document}